\documentclass[12pt]{amsart}    

\usepackage[margin=1in]{geometry}  
\usepackage[all]{xy}    
\usepackage{amsmath,amsthm,amssymb,mathrsfs,mathtools,bm,eucal,tensor}                  
\usepackage{bbm}
\usepackage[colorlinks]{hyperref}
\hypersetup{linkcolor=black, citecolor=black}

\usepackage{fancyhdr}
\usepackage[utf8]{inputenc}

\usepackage{chngcntr}
\numberwithin{equation}{subsection}
\newcounter{keepeqno}
\newenvironment{num}
 {\setcounter{keepeqno}{\value{equation}}%
  \begin{list}{(\theequation)}{\usecounter{equation}}%
  \setcounter{equation}{\value{keepeqno}}}
 {\end{list}}

\newcommand{\BB}{{\mathbb {B}}}
\newcommand{\BC}{{\mathbb {C}}}

\newcommand{\BR}{{\mathbb {R}}}

\newcommand{\BT}{{\mathbb {T}}}

\newcommand{\BZ}{{\mathbb {Z}}}

\newcommand{\CC}{{\mathcal {C}}}
\newcommand{\CE}{{\mathcal {E}}}
\newcommand{\CF}{{\mathcal {F}}}

\newcommand{\CI}{{\mathcal {I}}}

\newcommand{\CO}{{\mathcal {O}}}

\newcommand{\CS}{{\mathcal {S}}}

\newcommand{\CU}{{\mathcal {U}}}

\newcommand{\CX}{{\mathcal {X}}}

\newcommand{\FD}{{\mathfrak {D}}}

\newcommand{\Fa}{{\mathfrak {a}}}

\newcommand{\Fg}{{\mathfrak {g}}}
\newcommand{\Fh}{{\mathfrak {h}}}

\newcommand{\Fk}{{\mathfrak {k}}}
\newcommand{\Fl}{{\mathfrak {l}}}

\newcommand{\Fp}{{\mathfrak {p}}}

\newcommand{\Fs}{{\mathfrak {s}}}
\newcommand{\Ft}{{\mathfrak {t}}}

\newcommand{\RD}{{\mathrm {D}}}
\newcommand{\RE}{{\mathrm {E}}}
\newcommand{\RF}{{\mathrm {F}}}

\newcommand{\RI}{{\mathrm {I}}}
\newcommand{\RJ}{{\mathrm {J}}}

\newcommand{\RS}{{\mathrm {S}}}

\newcommand{\RU}{{\mathrm {U}}}

\newcommand{\RX}{{\mathrm {X}}}

\newcommand{\Ad}{{\mathrm{Ad}}}

\newcommand{\el}{{\mathrm{ell}}}

\newcommand{\Gal}{{\mathrm{Gal}}}

\newcommand{\inv}{{\mathrm{inv}}}

\newcommand{\reg}{{\mathrm{reg}}}

\newcommand{\SL}{{\mathrm{SL}}}

\newcommand{\sgn}{{\mathrm{sgn}}}

\newcommand{\st}{{\mathrm{st}}}

\newcommand{\ud}{\,\mathrm{d}}

\newcommand{\wt}{\widetilde}
\newcommand{\wh}{\widehat}

\newcommand{\bs}{\backslash}

\def\alp{{\alpha}}

\def\Del{{\Delta}}

\def\eps{{\epsilon}}

\def\sym{{\rm sym}}

\def\ome{{\omega}}

\def\Sig{{\Sigma}}
\def\gam{{\gamma}}
\def\Gam{{\Gamma}}

\def\wb{\overline} 

\def\vphi{\varphi}

\def\p{\prime}

\def\der{\mathrm{der}}

\def\ad{\mathrm{ad}}

\def\h{\hat}
\def\sc{\mathrm{sc}}
\def\der{\mathrm{der}}

\def\Gstarreg{G^*\text{-}\mathrm{reg}}

\def\RJ{\mathrm{J}}
\def\RF{\mathrm{F}}

\newtheorem{thm}{Theorem}[subsection]
\newtheorem{defin}[thm]{Definition}

\newtheorem{pro}[thm]{Proposition}
\newtheorem{lem}[thm]{Lemma}
\newtheorem{cor}[thm]{Corollary}

\newtheorem{rmk}[thm]{Remark}

\makeatletter

\newcommand{\Rmnum}[1]{\expandafter\@slowromancap\romannumeral #1@}
\makeatother

\begin{document}

\title[Compatibility between endoscopic transfer and Fourier transform over $\BR$]{
Fourier transform and endoscopic transfer on real Lie algebras}

\author{Cheng Chen}
\address{IMJ-PRG, CNRS\\ Bâtiment Sophie Germain, 8 place Aurélie Nemours, Paris, 75013, France}
\email{cheng.chen@imj-prg.fr}

\author{Zhilin Luo}
\address{Department of Mathematics\\
Purdue University\\
West Lafayette IN, 47907, USA}
\email{luo642@purdue.edu}


\date{\today}


\begin{abstract}
We prove that the endoscopic transfer on real Lie algebras exists and commutes with the Fourier transform, using methods that are purely local.
\end{abstract}

\maketitle

\tableofcontents

\section{Introduction}\label{sec:intro}

In his seminal work (\cite{waldtransfert}), J.-L. Waldspurger establishes a fundamental connection between smooth transfer and the Fourier transform on reductive Lie algebras and their endoscopic datum. Notably, the fundamental lemma for the unit element in the Lie algebra for almost all local places, now a theorem thanks to B.-C. Ngô in positive characteristic (\cite{MR2653248}) and Waldspurger's reduction to characteristic zero when the residual characteristic is sufficiently large (\cite{MR2241929}), implies both the existence of smooth transfer and its compatibility with the Fourier transform for Lie algebras over $p$-adic fields. The proof is of global nature and relies on the Lie algebra trace formula. Conversely, the work of D. Kazhdan and Y. Varshavsky (\cite[Remark~4.1.4]{MR2891532}) demonstrates that the compatibility between the Fourier transform and endoscopic transfer on Lie algebras implies the fundamental lemma for the unit element. D. Kazhdan and A. Polishchuk also delve into some variants of Waldspurger's theorem (\cite{MR1995380}).

The goal of the present paper is to establish the archimedean analogue of the theorem of Waldspurger using purely local methods. Over the complex field, this result is already known from Waldspurger (\cite[\S 9]{waldtransfert}), where the proof is likewise entirely local.

To state the theorems of this paper, we first fix our notation and conventions.

\subsubsection{Notation and conventions}

Let $G$ be a reductive algebraic group over $\BR$ with Lie algebra $\Fg$. 
We write $G$ and $\Fg$ for $G(\BC)$ and $\Fg(\BC)$ whenever there is no ambiguity.
Let $\Gam=\Gal_{\BC/\BR}$ be the absolute Galois group of $\BR$. Let $\CS(\Fg(\BR))$ (resp. $\CC^\infty_c(\Fg(\BR))$) denote the space of Schwartz (resp. smooth and compactly supported) functions on $\Fg(\BR)$, where $\Fg(\BR)$ is equipped with the Euclidean topology. 

Fix a non-degenerate, conjugation-invariant, symmetric bilinear form $\langle\cdot,\cdot\rangle_\Fg$ on $\Fg(\BR)$. Let $\psi(\langle\cdot,\cdot\rangle_\Fg)$ denote the Fourier kernel on $\Fg(\BR)$, where $\psi(x) = \exp(2\pi i x)$ for any $x\in \BR$. Define 
\[
\CF_{\psi_\Fg}(f)(X) = 
\int_{\Fg(\BR)}
f(Y)
\psi(\langle X,Y \rangle_\Fg)
\ud_\Fg Y,\quad f\in \CS(\Fg(\BR)),
\]
where $\ud_\Fg Y$ is the Haar measure on $\Fg(\BR)$ that is self-dual with respect to the Fourier transform. We equip $G(\BR)$ with the Haar measure $\ud g$ induced from $\ud_\Fg X$, so that the Jacobian of the exponential map is one at the identity.

Let $\Fg_\reg$ denote the regular semisimple locus of $\Fg$. Let $\Gam(\Fg_\reg(\BR)) = \Fg_\reg(\BR)/G(\BR)$ be the set of $G(\BR)$-conjugacy classes in $\Fg_\reg(\BR)$, and let $\Sig(\Fg_\reg(\BR)) = \Fg_\reg(\BR)/G(\BC)$ be the set of $G(\BR)$-stable conjugacy classes. Let $\Fs_G:\Gam(\Fg_\reg(\BR))\to \Sig(\Fg_\reg(\BR))$ be the natural projection. For $X\in \Fg_\reg(\BR)$, let $\CO_{X}\in \Gam(\Fg_\reg(\BR))$ (resp. $\CO^\st_{X}\in \Sig(\Fg_\reg(\BR))$) be its $G(\BR)$-conjugacy (resp. stable conjugacy) class. Let $\Fg(\BR)_\el\subset \Fg_\reg(\BR)$ be the set of elliptic regular semi-simple elements, i.e., elements $X\in \Fg_\reg(\BR)$ whose centralizer of $X$ in $G(\BR)$ is compact modulo the center of $G(\BR)$. 

For $X\in \Fg_\reg(\BR)$, let $G_X = T_X$ be its centralizer in $G$. Consider the invariant orbital integral over the orbit $\CO_X$:
\begin{equation}\label{eq:1:intro}
\RJ_G(X,f) = 
\RD^G(X)^{1/2}\int_{T_X(\BR)\bs G(\BR)}
f(g^{-1}Xg)\ud \wb{g},\quad f\in \CS(\Fg(\BR)),
\end{equation}
where $\RD^G(X) = |\det(\ad X|_{\Fg/\Fg_X})|$ is the absolute value of the Weyl discriminant of $X$. The integral \eqref{eq:1:intro} converges absolutely and defines a tempered distribution on $\Fg(\BR)$ (\cite[p.37,~Lemma~2]{MR0473111}). 

Let $\wh{j}^G(X,\cdot)$ be the Fourier transform of $\RJ_G(X,\cdot)$:
$$
\RJ_G(X,\CF_{\psi_\Fg}(f))=
\int_{\Fg(\BR)}
f(Y)
\wh{j}^G(X,Y)\ud_\Fg Y,\quad f\in \CS(\Fg(\BR)).
$$
Following \cite[p.171]{waldtransfert}, set $\wh{i}^G(X,Y) = \RD^G(Y)^{1/2}\wh{j}^G(X,Y)$. The distribution $(X,Y)\in \Fg(\BR)\times \Fg(\BR)\mapsto \wh{j}^G(X,Y)$ is locally integrable and smooth on $\Fg_\reg(\BR)\times \Fg_\reg(\BR)$. Moreover, $(X,Y)\in \Fg(\BR)\times \Fg(\BR)\mapsto \wh{i}^G(X,Y)$ is globally bounded (\cite[p.112,~Proposition~9]{MR0473111}). 

\subsubsection{Endoscopic transfer}

Let $\CE(\BR)$ be the set of tuples 
\[
\RE=(G,G^*,\vphi,H,s,\xi)    
\]
as in \cite[\S 1.1]{waldtransfert}, where $G,G^*$ are reductive algebraic groups over $\BR$ with $G^*$ quasi-split, $\vphi:G\to G^*$ is an inner twist, and $(H,s,\xi)$ is an endoscopic triple for $G^*$ in the sense of \cite[\S 7]{MR0757954}. Following \cite[\S 2.3]{waldtransfert}, one can define the transfer factor 
\[
\Del(X_H,X_G)    
\]
for $(X_H,X_G)\in \Fh_{G^*\text{-}\mathrm{reg}}(\BR)\times \Fg_\reg(\BR)$ where $\Fh_{\Gstarreg}$ denotes the set of $G^*$-regular semisimple elements of $\Fh_\reg$, and $\Del(\cdot,\cdot)$ is the analogue of the Langlands-Shelstad transfer factor on Lie algebras (\cite{transferfactorLS}).

For $f_G\in \CS(\Fg(\BR))$ and $X_H\in \Fh_{\Gstarreg}(\BR)$, set 
$$
\RJ_{G,H}(X_H,f_G) = 
\sum_{\CO_{X_G}\in \Gam(\Fg_\reg(\BR))}
\Del(X_H,X_G)
\RJ_G(X_G,f_G),
$$
where the summation is taken over all $X_G\in \Fg_\reg(\BR)$, modulo $G(\BR)$-conjugacy, such that $\Del(X_H,X_G)\neq 0$.

Similarly, for $f_H\in \CS(\Fh(\BR))$ and $\CO^\st_{X_H}\in \Sig(\Fh_{\Gstarreg}(\BR))$, set 
$$
\RJ_H^\st(X_H,f_H) = 
\sum_{\CO_{X^\p_H}\in \Fs_H^{-1}(\CO^\st_{X_H})}
\RJ_H(X^\p_H,f_H),
$$
where the summation is taken over the $H(\BR)$-conjugacy classes contained in the stable conjugacy class of $X_H$. 

\begin{thm}[Smooth transfer on Lie algebras]\label{thm:intro:ST}
For any $f_G\in \CS(\Fg(\BR))$ $($resp. $\CC^\infty_c(\Fg(\BR))$$)$, there exists $f_H\in \CS(\Fh(\BR))$ $($resp. $\CC^\infty_c(\Fh(\BR))$$)$, such that 
$$
\RJ_{G,H}(X_H,f_G) = \RJ^\st_H(X_H,f_H)
$$
for any $X_H\in \Fh_{\Gstarreg}(\BR)$.
\end{thm}

In the group setting, by the work of D. Shelstad (\cite{MR0627641}\cite{MR0661206}), smooth transfer is known for $f_G\in \CC(G(\BR))$ and $f_H\in \CC(H(\BR))$, where $\CC(G(\BR))$ denotes the space of Harish-Chandra Schwartz functions on $G(\BR)$. This result is extended to $\CC^\infty_c(G(\BR))$ by A. Bouaziz (\cite{MR1296557}). Using the same strategies as Bouaziz and Shelstad, and relying on Bouaziz's spectral characterization of Lie algebra orbital integrals for both test and Schwartz functions (\cite{MR1248081}), Theorem \ref{thm:intro:ST} follows immediately.

\subsubsection{Fourier transform and endoscopic transfer are compatible}

Building upon the work of \cite[p.91]{MR1344131} and \cite{waldtransfert}, for any $X_H\in \Fh_{\Gstarreg}(\BR)$, 
we introduce the following two tempered distributions on $\Fg(\BR)$, which are locally integrable and determined by their values on $\Fg_\reg(\BR)$.

\begin{defin}\label{defin:intro}
For $(X_H,X_G)\in \Fh_{\Gstarreg}(\BR)\times \Fg_\reg(\BR)$, set 
\begin{align*}
\RD_{G,H}(X_H,X_G) &= \gam_\psi(\Fg)
\sum_{\CO_{X_G^\p}\in \Gam(\Fg_\reg(\BR))}
\Del(X_H,X_G^\p)\wh{i}^G(X_G^\p,X_G),
\\
\wt{\RD}_{G,H}(X_H,X_G) &= 
\gam_\psi(\Fh)
\sum_{
\substack{
\CO_{X_H^\p}\in \Fs^{-1}_H(\CO^\st_{X_H})
\\ 
\CO_{X_H^{\p\p}}\in \Gam(\Fh_{\Gstarreg}(\BR))
}
}
w(X_H^{\p\p})^{-1}
\Del(X_H^{\p\p},X_G)
\wh{i}^H(X_H^\p,X_H^{\p\p}),
\end{align*}
where 
\begin{itemize}
\item $\CO_{X_G^\p}\in \Gam(\Fg_\reg(\BR))$, i.e., $X_G^\p\in \Fg_\reg(\BR)$ modulo $G(\BR)$-conjugacy;

\item $\CO_{X_H^\p}\in \Fs_H^{-1}(\CO_{X_H}^\st)$, i.e., $X_H^\p$ runs over the $H(\BR)$-conjugacy classes in the stable conjugacy class of $X_H$; 

\item $X^{\p\p}_H\in \Fh_{\Gstarreg}(\BR)$ modulo $H(\BR)$-conjugacy;

\item $w(X^{\p\p}_H) = |\Fs_H^{-1}(\CO_{X^{\p\p}_H})|$, the number of $H(\BR)$-conjugacy classes within the stable conjugacy class of $X^{\p\p}_H$;

\item $\gam_\psi(\Fg)$ and $\gam_\psi(\Fh)$ are the Weil constants associated with $(\Fg(\BR),\langle \cdot,\cdot\rangle_\Fg)$ and $(\Fh(\BR),\langle \cdot,\cdot\rangle_\Fh)$, respectively $($\cite[\S 1]{jlgl2}$)$. Here the non-degenerate conjugation-invariant symmetric bilinear form $\langle \cdot,\cdot\rangle_\Fh$ on $\Fh(\BR)$ is induced from $\langle \cdot,\cdot \rangle_\Fg$ as in \cite[VIII.~6]{MR1344131}.
\end{itemize}
\end{defin}

We now state our main result, which can be viewed as the archimedean analogue of \cite[Conjecture~2,~p.91]{MR1344131}.

\begin{thm}\label{thm:intro:FTcompatible}
With the above notation, for $X_H\in \Fh_{\Gstarreg}(\BR)$ and any $X_G\in \Fg_\reg(\BR)$, we have the identity
\begin{equation}\label{eq:intro:2}
\RD_{G,H}(X_H,X_G) = \wt{\RD}_{G,H}(X_H,X_G).
\end{equation}
\end{thm}

Following the same argument as \cite[Remarques.~(10)~p.92]{MR1344131}, Theorem \ref{thm:intro:FTcompatible} immediately implies the compatibility of the Fourier transform with smooth transfer.

\begin{cor}\label{cor:intro:FTcompatible}
With the above notation, let $f_G\in \CS(\Fg(\BR))$ and $f_H\in \CS(\Fh(\BR))$ satisfy the smooth transfer condition
$$
\RJ_{G,H}(X_H,f_G) = \RJ^\st_H(X_H,f_H),\quad X_H\in \Fh_{\Gstarreg}(\BR).
$$
Then
$$
\gam_{\psi}(\Fg)
\RJ_{G,H}(X_H,\CF_{\psi_\Fg}(f_G)) = 
\gam_\psi(\Fh)
\RJ^\st_H(X_H,\CF_{\psi_\Fh}(f_H)).
$$
\end{cor}

The above corollary immediately implies the following proposition (\cite[Proposition]{Wal00}), which is expected by experts.

\begin{pro}\label{pro:intro:FTofStableStillstable}
The Fourier transform of a stable distribution on $\Fg(\BR)$ is stable.
\end{pro}

We intend to use Proposition \ref{pro:intro:FTofStableStillstable} to study stable distributions supported on the nilpotent cone of $\Fg(\BR)$, following Waldspurger’s work in the $p$-adic setting.

Over a $p$-adic field, Waldspurger shows that \eqref{eq:intro:2} for all pairs $(G,H)$ implies the existence of smooth transfer for test functions on Lie algebras (\cite[Remarques~(8), p.~91]{MR1344131}). It would be interesting to know whether an analogous statement holds in the archimedean case.

On the other hand, following the global approach of Waldspurger, P.-H. Chaudouard establishes a weighted version of Waldspurger’s theorem over a $p$-adic field (\cite{MR2164623}), assuming the weighted fundamental lemma for Lie algebras that is known for split groups thanks to the work of P.-H. Chaudouard and G. Laumon (\cite{CL10, CL12}). It would be interesting to investigate whether a weighted version of Corollary~\ref{cor:intro:FTcompatible} can also be established in the archimedean case, again using purely local methods.

We now briefly outline the proof of Theorem~\ref{thm:intro:FTcompatible}:

\begin{itemize}
    \item
First, following the approach of~\cite[\S 6]{waldtransfert}, 
    one reduces the proof of~\eqref{eq:intro:2} to the case in which the endoscopic triple $(H,s,\xi)$ is elliptic and 
    $X_H \in \Fh(\BR)_\el$;

    \item
Next, using a result of W.~Rossmann~\cite[Corollary~(15), p.~217]{MR0508985}, we verify the identity explicitly when $X_G$ is also elliptic in $\Fg_\reg(\BR)$;

    \item
Finally, we show that, as invariant tempered distributions on $\Fg(\BR)$, 
    both $\RD_{G,H}(X_H,\cdot)$ and $\wt{\RD}_{G,H}(X_H,\cdot)$ 
    are $\RI(\Fg)$-finite and of elliptic type whenever 
    $X_H \in \Fh_{G^*\text{-}\mathrm{reg}}(\BR) \cap \Fh(\BR)_\el$, 
    where $\RI(\Fg)$ is the symmetric algebra of $\Fg(\BC)$ identified as the algebra of constant coefficient invariant differential operators on $\Fg(\BC)$. 
    By~\cite[Theorem~13, p.~116]{MR0473111}, 
    such distributions are uniquely determined by their values on the elliptic locus, 
    which completes the proof.
\end{itemize}
We emphasize that our proof is purely local, 
whereas Waldspurger’s $p$-adic proof, which is significantly deeper, is global in nature, 
relying on the fundamental lemma for the unit element at almost all local places.

\subsubsection{Organization}

The paper is organized as follows. 

In Section~\ref{sec:endoscopicSchwartz}, 
we review the existence of endoscopic transfer for test functions and Schwartz functions on Lie algebras over $\BR$.  
Subsection~\ref{subsec:endoscopic_datum} recalls the notion of endoscopic datum.  
Subsection~\ref{subsec:OrbSt} reviews the spectral characterization of orbital integrals and stable orbital integrals 
for both test and Schwartz functions on $\Fg(\BR)$.  
Subsection~\ref{subsec:smooth-transfer} states the existence of  endoscopic transfer in this context.

In Section~\ref{sec:waldspurgerid}, 
we establish Theorem~\ref{thm:intro:FTcompatible}, the real analogue of Waldspurger’s identity.  
Subsection~\ref{subsec:elliptic_situation} treats the elliptic case, 
where the endoscopic datum is elliptic, $X_H$ is $G^*$-elliptic, and $X_G \in \Fg(\BR)_\el$.  
Subsection~\ref{subsec:harish_chandra_s_uniqueness_theorem} recalls Harish-Chandra uniqueness theorem, 
which asserts that any $\RI(\Fg)$-finite invariant tempered distribution of elliptic type on $\Fg(\BR)$ 
is uniquely determined by its values on the elliptic locus.  
We then show that both $\RD_{G,H}(X_H,\cdot)$ and $\wt{\RD}_{G,H}(X_H,\cdot)$ satisfy these conditions, 
thereby completing the proof of Theorem~\ref{thm:intro:FTcompatible}.

\subsubsection{Acknowledgements}

The authors thank Professor P.-H. Chaudouard for inspiring discussions and Professor Y. Sakellaridis for his interest and comments. We also thank the anonymous referee for helpful comments.
The first author was supported by the European Union’s Horizon 2020 research and innovation programme under the Marie Skłodowska-Curie grant agreement No. 101034255. The second author was supported in part by research funds from the Dickson Instructorship at the University of Chicago and by start-up funds from Purdue University.

\section{Endoscopic transfer}\label{sec:endoscopicSchwartz}

In this section, we review the existence of endoscopic transfer 
for both test functions and Schwartz functions on Lie algebras over $\BR$. 
These results follow directly by combining the work of 
Shelstad (\cite{shelsteadinner}) and 
Bouaziz (\cite{MR1248081,MR1296557}).

\subsection{Endoscopic datum}\label{subsec:endoscopic_datum}

In this subsection, we briefly recall the definitions of endoscopic datum 
and transfer factors on Lie algebras, 
following \cite[\S 2]{waldtransfert}, \cite{kottwitztransfer}, and \cite{transferfactorLS}.

\subsubsection{Endoscopic tuples}

Let $\CE(\BR)$ be the set of tuples 
$
\RE = (G,G^*,\vphi,H,s,\xi)    
$ where:
\begin{enumerate}
    \item $\vphi \colon G \to G^*$ is an inner twist between two reductive algebraic groups over $\BR$, 
    with $G^*$ the quasi-split pure inner form of $G$;
    \item $(H, s, \xi)$ is an \emph{endoscopic triple} (\cite[\S 7]{MR0757954}), meaning that
    \begin{itemize}
        \item $H$ is a quasi-split reductive algebraic group over $\BR$;
        \item $\xi \colon \hat{H} \to \hat{G}$ is an embedding from the dual group of $H$ to that of $G$, and the $\hat{G}$-conjugacy class of $\xi$ is fixed by $\Gamma$;
        \item $s \in Z_{\hat{H}}$, the center of the dual group $\hat{H}$, satisfies:
        \begin{itemize}
            \item $\xi(\hat{H})$ is the connected component of the centralizer of $\xi(s)$ in $\hat{G}$;
            \item the image of $s$ in $Z_{\hat{H}} / \xi^{-1}( Z_{\hat{G}} )$ is fixed by $\Gamma$;
            \item its image $\bar{s}$ in
            $
            \pi_0 \big( [ Z_{\hat{H}} / \xi^{-1}( Z_{\hat{G}} ) ]^\Gamma \big)
            $
            lies in the kernel of the natural map
            \[
            \pi_0 \big( [ Z_{\hat{H}} / \xi^{-1}( Z_{\hat{G}} ) ]^\Gamma \big)
            \longrightarrow H^1(\BR, Z_{\hat{G}}).
            \]
        \end{itemize}
    \end{itemize}
\end{enumerate}

Fix $\Gamma$-stable Borel pairs
$
(\BB,\BT) \text{ of } G^*,
(\BB_H,\BT_H) \text{ of } H,
(\hat{\BB},\hat{\BT}) \text{ of } \hat{G},\text{ and }
(\hat{\BB}_H,\hat{\BT}_H) \text{ of } \hat{H}.
$
Since Borel pairs are unique up to conjugation, 
we may assume that 
$
\xi(\hat{\BT}_H) = \hat{\BT}
$ and 
$
\xi(\hat{\BB}_H) \subset \hat{\BB}.
$ The restriction $\xi|_{\hat{\BT}_H}$ induces a dual isomorphism 
$
\eta \colon \BT_H \to \BT
$
 over $\BC$.

\begin{rmk}\label{rmk:admissible_embed_Weylequivariant}
Since $\hat{H} \subset \hat{G}$, 
we have $W^{\hat{H}} \subset W^{\hat{G}}$, 
where $W^{\hat{H}}$ and $W^{\hat{G}}$ are the absolute Weyl groups of $\h{H}$ and $\h{G}$, respectively. 
Hence, the isomorphism 
\(\xi|_{\hat{\BT}_H} \colon \hat{\BT}_H \to \hat{\BT}\) 
is equivariant with respect to the actions of
\(W^{\hat{H}} \hookrightarrow W^{\hat{G}}\).  
Passing to the dual side, and using the fact that 
\(W^{H} \simeq W^{\hat{H}} \hookrightarrow W^{\hat{G}} \simeq W^G\), 
we see that 
\(\eta \colon \BT_H \to \BT\) 
is also equivariant with respect to the actions of 
\(W^H \hookrightarrow W^G\).
\end{rmk}

\subsubsection{$G^*$-regular elements}

\begin{defin}\label{defin:Gstarreg}
An element $X_H \in \Fh_\reg$ is called \emph{$G^*$-regular} 
if there exists $h \in H$ such that
\[
\Ad_h(X_H) \in \Ft_H
\quad \text{and} \quad
\eta\big(\Ad_h(X_H)\big) \in \Fg^*_\reg,
\]
where $\Ft_H$ is the Lie algebra of $\BT_H$, 
and $\Ad$ denotes the adjoint action.  
Let $\Fh_{\Gstarreg}$ be the set of $G^*$-regular elements in $\Fh$. 

\noindent Fix $(X_H, X_G) \in \Fh_{\Gstarreg}(\BR) \times \Fg_\reg(\BR)$.  
A \emph{diagram} 
\[
\FD(X_H, X_G; \BR)
\]
consists of data
\begin{equation}\label{eq:diagram}
\begin{cases}
h \in H, \quad g^* \in G^*_\sc, \quad g \in G_\sc, \\[4pt]
\xymatrix{
T_{X_H} \ar[r]^{\Ad_h} 
\ar@/^3pc/[rrrrrr]_{\CI^{X_G}_{X_H}}
\ar@/_1.5pc/[rrr]_{\CI^{X_{G^*}}_{X_H}}
& \BT_H \ar[r]^{\eta} & \BT \ar[r]^{\Ad_{g^*}} &
T_{X_{G^*}} 
\ar[rrr]_{\Ad_g \circ \vphi^{-1} = \CI^{X_G}_{X_{G^*}}} &&& T_{X_G}
}
\end{cases}
\end{equation}
satisfying:
\begin{enumerate}
    \item The morphisms in \eqref{eq:diagram} are isomorphisms over $\BC$, and the diagram is commutative;
    \item Both 
    \(\CI^{X_{G^*}}_{X_H} = \Ad_{g^*} \circ \eta \circ \Ad_h\) 
    and 
    \(\CI^{X_G}_{X_{G^*}} = \Ad_g \circ \vphi^{-1}\) 
    are defined over $\BR$;
    \item Let 
    \(\CI^{X_G}_{X_H} = \CI^{X_G}_{X_{G^*}} \circ \CI^{X_{G^*}}_{X_H}\).  
    Then, on Lie algebras, \(\CI^{X_G}_{X_H}(X_H) = X_G\).
\end{enumerate}
Here $G_\sc$ $($resp. $G^*_\sc$$)$ is the simply connected cover of 
$G_\der$ $($resp. $G^*_\der$$)$.
\end{defin}

Following this notation, we write
\[
X_{G^*} = \Ad_{g^*}\Big( \eta\big( \Ad_h(X_H) \big) \Big).
\]
In the sequel, we do not distinguish between the isomorphisms 
$\CI^*_* $ and their Lie algebra versions.

The isomorphism $\Ad_{g^*} \colon \BT \to T_{X_{G^*}}$ 
induces an isomorphism 
\[
\hat{i}_{g^*} \colon \hat{\BT} \to \hat{T}_{X_{G^*}}.
\]
The image of $s \in Z_{\hat{H}}$ under the composition 
$
\hat{i}_{g^*} \circ \xi \colon Z_{\h{H}} \to \hat{T}_{X_{G^*}}
$ and the natural projection 
\(\hat{T}_{X_{G^*}} \to \hat{T}_{X_{G^*}} / Z_{\hat{G}}\) 
is $\Gamma$-invariant and lies in the kernel of the map induced by the long exact sequence:
\[
K(T_{X_{G^*}}/\BR) := 
\ker\Big(
\pi_0
\big(
[\hat{T}_{X_{G^*}} / Z_{\hat{G}}]^\Gamma
\big)
\longrightarrow H^1(\BR, Z_{\hat{G}})
\Big)
.
\]
We denote this image by $\kappa_{X_H}$.

Finally, let 
\(\FD(X_H,X_G;\BR)\) and 
\(\FD(\bar{X}_H,\bar{X}_G;\BR)\) 
be two diagrams as in Definition~\ref{defin:Gstarreg}.  
We now recall the definition of the transfer factor 
\(\Del(X_H,X_G)\) 
from \cite[\S 2.3]{waldtransfert}.

\subsubsection{$\Del_\RI$}

Fix a $\Gamma$-stable splitting of $G^*$
\[
(\BB, \BT, \{ X_\alpha \}_{\alpha \in \Delta_{G^*}}).
\]
Here $\Delta_{G^*}$ denotes the set of simple roots of $\BT$ relative to $\BB$, 
and each $X_\alpha \neq 0$ lies in the root space of $\alpha$ and satisfies
\[
\sigma(X_\alpha) = X_{\sigma(\alpha)}, \quad \sigma \in \Gamma.
\]

Let $\Sigma(T_{X_{G^*}})$ be the set of roots of $T_{X_{G^*}}$ in $\Fg^*$, 
and let $\Sigma^+(T_{X_{G^*}})$ be the subset of positive roots 
with respect to $\Ad_{g^*}(\BB)$.  
For $\sigma \in \Gamma$, let $n_\sigma$ be the image of 
\((g^{*})^{-1} \sigma(g^*) \in W^{G^*_\sc}\),  
where $W^{G^*_\sc}$ is the absolute Weyl group of $G^*_\sc$.  
Let $T_{X_{G^*_\sc}}$ be the preimage of $T_{X_{G^*}}$ in $G^*_\sc$.

For any $\alpha \in \Sigma(T_{X_{G^*}})$, 
choose $a_\alpha \in \BC^\times$ satisfying
\[
\sigma(a_\alpha) = a_{\sigma(\alpha)}, 
\qquad
a_{-\alpha} = -a_\alpha,
\]
where $\sigma \in \Gamma$ is the nontrivial involution.  
Define
\[
a_\sigma = 
\sum_{\substack{
\alpha \in \Sigma^+(T_{X_{G^*}}) \\
\sigma^{-1}(\alpha)<0
}}
\check{\alpha} \otimes a_\alpha
\ \in \
\RX_*(T_{X_{G^*_\sc}}) \otimes_\BZ \BC^\times 
\ \simeq \ T_{X_{G^*_\sc}}(\BC),
\]
where $\RX_*(T_{X_{G^*_\sc}})$ is the cocharacter lattice of $T_{X_{G^*_\sc}}$.  Set
\[
\tau(\sigma) 
= a_\sigma \, g^* \, n_\sigma \, \sigma\big( (g^{*})^{-1} \big).
\]
Then the map 
\(\sigma \mapsto \tau(\sigma)\) 
defines a $1$-cocycle in $T_{X_{G^*_\sc}}$.  
Let $\tau \in H^1(\BR, T_{X_{G^*_\sc}})$ be its cohomology class,  
which is dual to \( K(T_{X_{G^*}}/\BR) \).  

We define
\begin{equation}\label{eq:defin:Del_I}
\Del_\RI(X_H, X_G) = \langle \tau, \kappa_{X_H} \rangle. 
\end{equation}

\begin{rmk}\label{rmk:Del_I_dependence}
The factor $\Del_\RI(X_H, X_G)$ depends only on the isomorphism 
\[
\Ad_{g^*} \colon \BT \xrightarrow{\sim} T_{X_{G^*}},
\]
and not on the particular regular semisimple element 
$X_{G^*}$ contained in $T_{X_{G^*}}$.
\end{rmk}

\subsubsection{$\Del_{\RI\RI}$}

Define
\[
\Sigma^\sym(T_{X_{G^*}})
=
\big\{
\alpha \in \Sigma(T_{X_{G^*}}) 
\ \big|\ 
\exists\, \sigma \in \Gamma,\ \sigma(\alpha) = -\alpha
\big\}.
\]
For $\alpha \in \Sigma^\sym(T_{X_{G^*}})$, set
\[
\Gamma_\alpha = \{ \sigma \in \Gamma \mid \sigma(\alpha) = \alpha \} = \{ 1 \},
\qquad
\Gamma_{\pm \alpha} = \{ \sigma \in \Gamma \mid \sigma(\alpha) = \pm \alpha \} = \Gamma.
\]
Let $\BR_\alpha = \BC$ and $\BR_{\pm\alpha} = \BR$ be the fixed fields of 
$\Gamma_\alpha$ and $\Gamma_{\pm\alpha}$, respectively.  
Let $\chi_\alpha = \sgn$ be the sign (=quadratic) character of 
$\BR^\times_{\pm\alpha} \simeq \BR^\times$.

The isomorphism $T_{X_H} \xrightarrow{\sim} T_{X_{G^*}}$ 
identifies $\Sigma(T_{X_H})$ as a subset of $\Sigma(T_{X_{G^*}})$.  
Define
\[
\Sigma^\sym_{\text{hors }H}(T_{X_{G^*}})
=
\Sigma^\sym(T_{X_{G^*}})
\setminus
\big(
\Sigma(T_{X_H}) \cap \Sigma^\sym(T_{X_{G^*}})
\big),
\]
that is, the set of symmetric roots of $T_{X_{G^*}}$ not coming from $T_{X_H}$.

We now define
\begin{equation}\label{eq:defin:Del_II}
\Del_{\RI\RI}(X_H, X_G)
=
\prod_{\alpha}
\chi_\alpha
\Big(
\frac{\alpha(X_{G^*})}{a_\alpha}
\Big),
\end{equation}
where the product is taken over a set of representatives of 
$\Gamma$-orbits in 
$\Sigma^\sym_{\text{hors }H}(T_{X_{G^*}})$.

\subsubsection{$\Del_{\RI\RI\RI}$}

Define
\[
U 
= \frac{
T_{X_{G^*_\sc}} \times T_{\bar{X}_{G^*_\sc}}
}{
\{ (z, z^{-1}) \mid z \in Z_{G^*_\sc} \}
}.
\]
For $\sigma \in \Gamma$ with 
$\sigma(\vphi) \circ \vphi^{-1} = \Ad_{u_\sigma}$ 
for some $u_\sigma \in G^*_\sc$,  
consider the map
\[
\sigma \longmapsto
\big(
u_\sigma \,\vphi\big(\sigma(g^{*,-1}) g^*\big), \ 
\vphi\big(\bar{g}^{*,-1} \sigma(\bar{g}^*)\big)\, u_\sigma^{-1}
\big)
\in U,
\]
which defines a $1$-cocycle in $U$. Let
\[
\inv(X_H, X_G; \bar{X}_H, \bar{X}_G) 
\in H^1(\BR, U)
\]
denote the cohomology class of this cocycle.

Next, let 
\(\hat{T}_{X_{G^*_\sc}}\) and \(\hat{T}_{\bar{X}_{G^*_\sc}}\) 
be the inverse images of 
\(\hat{T}_{X_{G^*}}\) and \(\hat{T}_{\bar{X}_{G^*}}\) in \(\hat{G}_\sc\).  
Set
\[
\hat{U} 
= 
\frac{
\hat{T}_{X_{G^*_\sc}} \times \hat{T}_{\bar{X}_{G^*_\sc}}
}{
\{ (z, z) \mid z \in Z_{\hat{G}_\sc} \}
}.
\]
Fix $\hat{s} \in \hat{\BT}_\sc$ 
lifting $\xi(s) \in \hat{\BT} / Z_{\hat{G}}$.  
The isomorphisms $\hat{i}_{g^*}$ and $\hat{i}_{\bar{g}^*}$ 
provide
\(\hat{\BT}_\sc \simeq \hat{T}_{\bar{X}_{G^*_\sc}}\).  
Let
\[
s_U \in \pi_0\big( \hat{U}^\Gamma \big)
\]
be the image of 
\(( \hat{i}_{g^*}(\hat{s}),\ \hat{i}_{\bar{g}^*}(\hat{s}) )\).

The groups $H^1(\BR, U)$ and $\pi_0(\hat{U}^\Gamma)$ 
are naturally in Pontryagin duality.  
We define
\begin{equation}\label{eq:defin:Del_III}
\Del_{\RI\RI\RI}(X_H, X_G; \bar{X}_H, \bar{X}_G)
=
\big\langle 
\inv(X_H, X_G; \bar{X}_H, \bar{X}_G), \ 
s_U
\big\rangle.
\end{equation}
\begin{rmk}\label{rmk:Del_III_dependence}
The factor 
\(\Del_{\RI\RI\RI}(X_H, X_G; \bar{X}_H, \bar{X}_G)\) 
depends only on the isomorphisms
\[
\Ad_{g^*} \colon \BT \to T_{X_{G^*}}
\quad \text{and} \quad
\Ad_{\bar{g}^*} \colon \BT \to T_{\bar{X}_{G^*}},
\]
and not on the specific regular semisimple elements 
$X_{G^*}$ or $\bar{X}_{G^*}$ in these maximal tori.
\end{rmk}

\subsubsection{Transfer factor}

The transfer factor is defined by
\begin{equation}\label{eq:defin:Del-4}
\Del(X_H, X_G; \bar{X}_H, \bar{X}_G)
=
\frac{\Del_\RI(X_H, X_G)}{\Del_\RI(\bar{X}_H, \bar{X}_G)}
\frac{\Del_{\RI\RI}(X_H, X_G)}{\Del_{\RI\RI}(\bar{X}_H, \bar{X}_G)}
\Del_{\RI\RI\RI}(X_H, X_G; \bar{X}_H, \bar{X}_G),
\end{equation}
which is independent of all the auxiliary choices above.  

Fix a diagram $\FD(\bar{X}_H, \bar{X}_G; \BR)$  
and assign an arbitrary nonzero value to \(\Del(\bar{X}_H, \bar{X}_G)\).  
Then the normalized transfer factor is defined by
\begin{equation}\label{eq:defin:Del-transfer_factor}
\Del(X_H, X_G)
=
\begin{cases}
\Del(\bar{X}_H, \bar{X}_G)\,
\Del(X_H, X_G; \bar{X}_H, \bar{X}_G), 
& \text{if a diagram } \FD(X_H, X_G; \BR) \text{ exists}, \\[4pt]
0, & \text{otherwise.}
\end{cases}
\end{equation}

\subsection{Orbital integrals and stable orbital integrals}\label{subsec:OrbSt}

In this subsection, we recall the spectral characterization of Lie algebra orbital integrals 
and their stable variants over $\BR$, following 
\cite{MR1248081,MR1296557} and \cite{shelsteadinner}.

Let $G$ be a reductive algebraic group over $\BR$ with Lie algebra $\Fg$.  
When no confusion arises, we write $G = G(\BC)$ and $\Fg = \Fg(\BC)$.  

Fix a Cartan subalgebra $\Ft \subset \Fg$ defined over $\BR$.  
Let $\Sigma(\Ft)$ be the set of roots of $\Ft$ in $\Fg$, 
and let $\Sigma^+(\Ft) \subset \Sigma(\Ft)$ be a choice of positive roots.  
The Cartan involution on $G$ induces a decomposition
\[
\Ft = \Ft^+ \oplus \Ft^-,
\]
where $\Ft^\pm$ are the $\pm 1$-eigenspaces.
\begin{num}
\item
A root $\alpha \in \Sigma(\Ft)$ is called \textbf{imaginary} 
if $\alpha|_{\Ft^+} = 0$.  
Let $\Sigma_\mathrm{im}(\Ft) \subset \Sigma(\Ft)$ denote the set of imaginary roots.
\end{num}
For $\alpha \in \Sigma_\mathrm{im}(\Ft)$, let $H_\alpha$ be its coroot in $\Ft$.  
Then $\overline{H_\alpha} = - H_\alpha$.  
Let $X_\alpha$ be a root vector associated to $\alpha$,  
and choose $c \in \BC^\times$ so that the root vector 
\[
X_{-\alpha} = c \overline{X_\alpha}
\]
associated to $-\alpha$ satisfies 
\[
[X_\alpha, X_{-\alpha}] = H_\alpha.
\]  
Consider the real $3$-dimensional subspace and its embedding
\begin{equation}\label{eq:lifting-of-embedding}
\BR X_\alpha \;+\; \BR H_\alpha \;+\; \BR X_{-\alpha}
\ \hookrightarrow\ 
\Fg(\BR).
\end{equation}
\begin{itemize}
    \item 
    If $c > 0$, then the embedding 
    \eqref{eq:lifting-of-embedding} lifts to a homomorphism 
    \[
    \SL_2(\BR) \longrightarrow G(\BR),
    \]
    and we call $\alpha$ a \textbf{noncompact root}.  

    \item 
    If $c < 0$, then the embedding 
    \eqref{eq:lifting-of-embedding} lifts to a homomorphism 
    \[
    \RS\RU_2 \longrightarrow G(\BR),
    \]
    and we call $\alpha$ a \textbf{compact root}.
\end{itemize}

\subsubsection{Jump datum}

Let $X_0 \in \Fg(\BR)$ be a semisimple element such that 
the derived algebra of its centralizer 
$\Fg(\BR)_{X_0}$ in $\Fg(\BR)$ is isomorphic to $\Fs\Fl_2(\BR)$.  

Let $\Ft$ be a \textbf{fundamental Cartan subalgebra} of $\Fg_{X_0}$ over $\BR$,  
that is, under the Cartan decomposition
\[
\Fg(\BR)_{X_0} = \Fk_{X_0} \oplus \Fp_{X_0},
\]
the real subspace $\Ft(\BR) \cap \Fk_{X_0}$ 
has maximal dimension among all Cartan subalgebras of $\Fg(\BR)_{X_0}$.  

Let $\pm\alpha$ be the unique roots of $\Ft$ in $\Fg$ that vanish on $X_0$.  
Since $\Ft$ is fundamental, these roots satisfy 
\[
\pm \alpha \in \Sigma_\mathrm{im}(\Ft) 
\quad\text{and are noncompact.}
\]
Choose nonzero root vectors $X_{\pm \alpha}$ corresponding to $\pm \alpha$ such that
\[
[H_\alpha, X_{\pm \alpha}] = \pm 2 X_{\pm \alpha}, 
\qquad
[X_\alpha, X_{-\alpha}] = H_\alpha,
\qquad
\overline{X_\alpha} = X_{-\alpha}.
\]
Then the real subspace
\[
\ker \alpha \;\oplus\; i\BR(X_\alpha - X_{-\alpha})
\]
is the real part of a Cartan subalgebra $\Ft'$ of $\Fg_{X_0}$.

Define
\[
s = \exp\big( -\frac{i\pi}{4} \, \ad( X_\alpha + X_{-\alpha}) \big),
\]
which lies in the adjoint group of $\Fg$ and satisfies
\[
s \cdot \Ft = \Ft', 
\qquad 
s \cdot X_0 = X_0.
\]
The mapping
\begin{equation}\label{eq:jumpdatum:mapping}
\gamma \ \longmapsto\  \gamma \circ s^{-1}
\end{equation}
induces a bijection
\[
\Sigma_\mathrm{im}(\Ft) \cap \{\alpha\}^\perp
\ \xrightarrow{\sim}\ 
\Sigma_\mathrm{im}(\Ft'),
\]
where 
\[
\{\alpha\}^\perp = \{\, \beta \in \Sigma(\Ft) \mid \beta(H_\alpha) = 0 \,\}.
\]
Let $\Sigma^+_\mathrm{im}(\Ft) \subset \Sigma_\mathrm{im}(\Ft)$ 
be a chosen set of positive imaginary roots.  
Under the map \eqref{eq:jumpdatum:mapping}, 
\[
\Sigma^+_\mathrm{im}(\Ft) \cap \{\alpha\}^\perp
\quad \mapsto \quad
\Sigma^+_\mathrm{im}(\Ft') \subset \Sigma_\mathrm{im}(\Ft')
\]
defines a set of positive roots for $\Ft'$.

\begin{num}
\item\label{num:adapted_to_alpha}
Following \cite[p.24]{shelsteadinner}, 
we say that $\Sigma^+_\mathrm{im}(\Ft)$ is \textbf{adapted to $\alpha$}  
if it contains all imaginary roots $\beta \in \Sigma_\mathrm{im}(\Ft)$ 
with $\beta(H_\alpha) > 0$.
\end{num}

\begin{defin}\label{defin:jump-datum}
With the notation above, a \emph{jump datum} is the tuple
\[
\big(
X_0, 
\Ft, \Sigma^+_\mathrm{im}(\Ft), 
\Ft', \Sigma^+_\mathrm{im}(\Ft'), 
s
\big).
\]
\end{defin}

\subsubsection{Spectral characterization}

Let $\CU \subset \Fg(\BR)$ be a $G(\BR)$-invariant completely open neighborhood,  
meaning that $\CU$ contains the semisimple part of the Jordan decomposition of every element in $\CU$.  
Let $\CU_\reg$ denote the subset of regular semisimple elements in $\CU$.

Following \cite{MR1248081}, define $\CO\CC(\CU)$ (resp. $\CO\CS(\CU)$)  
as the space of $G(\BR)$-invariant functions  
\[
\RJ \in \CC^\infty(\CU_\reg)
\]
satisfying the following conditions:
\begin{itemize}
\item[\(I_1(\Fg)\):]  
Let $\Ft$ be a Cartan subalgebra of $\Fg$ over $\BR$,  
$K^\Ft_\CU \subset \Ft(\BR) \cap \CU$ a compact neighborhood,  
and $u\in \RI(\Ft)$ a constant-coefficient differential operator on $\Ft$.  
Let $\RJ_\Ft$ be the restriction of $\RJ$ to $\Ft(\BR)\cap \CU_\reg$.  
Then the function $\partial(u)\RJ_\Ft$ is bounded on  
$K^\Ft_\CU \cap \CU_\reg$.

\item[\(I_2(\Fg)\):]  
Let $\Ft$ be a Cartan subalgebra of $\Fg$ over $\BR$,  
and $\Sigma^+_\mathrm{im}(\Ft) \subset \Sigma_\mathrm{im}(\Ft)$  
a system of positive imaginary roots.  
Define
\[
\Fa_{\Sigma^+_\mathrm{im}(\Ft)} \colon 
X \mapsto 
\prod_{\alpha \in \Sigma^+_\mathrm{im}(\Ft)}
\frac{\alpha(X)}{|\alpha(X)|},
\quad
X \in \Ft(\BR) \cap \CU_\reg.
\]
Then $\Fa_{\Sigma^+_\mathrm{im}(\Ft)} \RJ_\Ft$  
extends smoothly to the complement of the imaginary noncompact roots  
in $\Ft(\BR)\cap \CU$.

\item[\(I_3(\Fg)\):]  
Let $(X_0, \Ft, \Sigma^+_\mathrm{im}(\Ft), \Ft', \Sigma^+_\mathrm{im}(\Ft'), s)$  
be a jump datum (Definition~\ref{defin:jump-datum}),  
and $u\in \RI(\Ft)$ a constant-coefficient differential operator on $\Ft$.  
Then
\begin{align*}
\partial(u)\big(\Fa_{\Sigma^+_\mathrm{im}(\Ft)} \RJ_\Ft\big)(X_0^+)
-
\partial(u)\big(\Fa_{\Sigma^+_\mathrm{im}(\Ft)} \RJ_\Ft\big)(X_0^-)
&=
i\, d(X_0) \,
\partial(s \cdot u)\big(\Fa_{\Sigma^+_\mathrm{im}(\Ft')} \RJ_{\Ft'}\big)(X_0),
\end{align*}
where $d(X_0) = 2$ or $1$ depending on whether the reflection for $\alpha$  
can be realized in $G_{X_0}$.  Here
\[
\partial(u)\big(\Fa_{\Sigma^+_\mathrm{im}(\Ft)} \RJ_\Ft\big)(X_0^\pm)
=
\lim_{t \to \pm 0}
\partial(u)\big(\Fa_{\Sigma^+_\mathrm{im}(\Ft)} \RJ_\Ft\big)(X_0 + it H_\alpha),
\]
and for sufficiently small $t \neq 0$,  
we have $X_0 \pm it H_\alpha \in \Ft_\reg(\BR)$,  
so the above limit is well-defined.

\item[\(I_4(\Fg)\):]  
This condition specifies the growth behavior:
\begin{itemize}
    \item[\((\CC)\)]  
    For $\RJ \in \CO\CC(\CU)$ and any Cartan $\Ft$,  
    the restriction $\RJ_\Ft$ is compactly supported on $\Ft(\BR)\cap \CU$.
    
    \item[\((\CS)\)]  
    For $\RJ \in \CO\CS(\CU)$ and any Cartan $\Ft$,  
    the restriction $\RJ_\Ft$ lies in the Schwartz space  
    $\CS(\Ft(\BR)\cap \CU_\reg)$,  
    with seminorms
    \[
    \nu_{P,u}(\RJ_\Ft)
    =
    \sup_{X\in \Ft(\BR)\cap \CU_\reg}
    \Big| P(X) (\partial(u)\RJ_\Ft)(X) \Big| < \infty,
    \]
    where $u\in \RI(\Ft)$ is a constant-coefficient differential operator on $\Ft$,  
    and $P$ is a polynomial function on $\Ft$.
\end{itemize}
\end{itemize}

\begin{thm}[\cite{MR1248081}]\label{thm:orb-spectral}
With the notation above, the map
\[
f \in \CC^\infty_c(\CU) 
\quad 
(\text{resp. } f \in \CS(\Fg(\BR)))
\quad \longmapsto\quad
\RJ_G(\cdot,f)|_{\CU_\reg}
\]
has image contained in $\CO\CC(\CU)$ $($resp. $\CO\CS(\CU)$$)$  
and is surjective.  
In particular, for $X \in \Fg_\reg(\BR)$,  
the distribution $\RJ_G(X,\cdot)$ is tempered.
\end{thm}
Similarly, following \cite{shelsteadinner} and \cite{MR1296557},  
we obtain the spectral characterization for stable orbital integrals.  

Let $\CS\CC(\CU)$ (resp. $\CS\CS(\CU)$)  
be the space of stably invariant functions  
\(\RJ^\st \in \CC^\infty(\CU_\reg)\)  
(i.e., constant on stable conjugacy classes in $\CU_\reg$)  
satisfying the same conditions $I_1(\Fg) = I^\st_1(\Fg)$, $I_2(\Fg) = I^\st_2(\Fg)$, $I_4(\Fg) = I^\st_4(\Fg)$ as above, except for a modified jump condition:

\begin{itemize}
\item[\(I^\st_3(\Fg)\):]  
For any jump datum 
$(X_0, \Ft, \Sigma^+_\mathrm{im}(\Ft), \Ft', \Sigma^+_\mathrm{im}(\Ft'), s)$  
and any constant-coefficient differential operator $u\in \RI(\Ft)$ on $\Ft$,  
we have
\begin{align*}
\partial(u)\big(\Fa_{\Sigma^+_\mathrm{im}(\Ft)} \RJ^\st_\Ft\big)(X_0^+)
-
\partial(u)\big(\Fa_{\Sigma^+_\mathrm{im}(\Ft)} \RJ^\st_\Ft\big)(X_0^-)
&=
2i \,
\partial(s\cdot u)\big(\Fa_{\Sigma^+_\mathrm{im}(\Ft')} \RJ^\st_{\Ft'}\big)(X_0).
\end{align*}
\end{itemize}

\begin{thm}[\cite{shelsteadinner}\cite{MR1296557}]\label{thm:orb-spectral-stable}
With the notation above, the map
\[
f \in \CC^\infty_c(\CU) 
\quad 
(\text{resp. } f \in \CS(\Fg(\BR)))
\quad \longmapsto\quad
\RJ_G^\st(\cdot,f)|_{\CU_\reg}
\]
lands in $\CS\CC(\CU)$ $($resp. $\CS\CS(\CU)$$)$ and is surjective.
\end{thm}

\subsection{Existence of the smooth transfer}\label{subsec:smooth-transfer}

In this subsection, we review the existence of the endoscopic smooth transfer  
for both test and Schwartz functions on $\Fg(\BR)$.  

\begin{thm}\label{thm:smooth-transfer}
Let $f_G \in \CC^\infty_c(\Fg(\BR))$ $($resp.~$f_G \in \CS(\Fg(\BR))$$)$.  
Then there exists $f_H \in \CC^\infty_c(\Fh(\BR))$ $($resp.~$f_H \in \CS(\Fh(\BR))$$)$  
such that
\[
\RJ^\st_H(X_H, f_H) = \RJ_{G,H}(X_H, f_G),
\quad
\forall\, X_H \in \Sigma(\Fh_{\Gstarreg}(\BR)).
\]
\end{thm}
To prove the theorem, it suffices to show that the function
\[
X_H \in \Sigma(\Fh_\reg(\BR))
\quad \longmapsto \quad
\begin{cases}
\RJ_{G,H}(X_H, f_G), & X_H \in \Sigma(\Fh_{\Gstarreg}(\BR)), \\[4pt]
0, & \text{otherwise},
\end{cases}
\]
belongs to
\(\CS\CC(\Fh(\BR))\) \( (\text{resp. } \CS\CS(\Fh(\BR))) \)  
for every $f_G \in \CC^\infty_c(\Fg(\BR))$ \( (\text{resp. } \CS(\Fg(\BR))) \).  

The proof is entirely analogous to \cite{MR532374} and is therefore omitted.

\section{Waldspurger's identity}\label{sec:waldspurgerid}

In this section, we are going to establish Theorem \ref{thm:intro:FTcompatible}, which we restate for convenience.

\begin{thm}\label{thm:Waldspurger-Id}
For $(X_H,X_G)\in \Fh_{\Gstarreg}(\BR)\times \Fg_\reg(\BR)$, set 
\begin{align*}
\RD_{G,H}(X_H,X_G) &= \gam_\psi(\Fg)
\sum_{\CO_{X_G^\p}\in \Gam(\Fg_\reg(\BR))}
\Del(X_H,X_G^\p)\wh{i}^G(X_G^\p,X_G),
\\
\wt{\RD}_{G,H}(X_H,X_G) &= 
\gam_\psi(\Fh)
\sum_{
\substack{
\CO_{X_H^\p}\in \Fs^{-1}_H(\CO^\st_{X_H})
\\ 
\CO_{X_H^{\p\p}}\in \Gam(\Fh_{\Gstarreg}(\BR))
}
}
w(X_H^{\p\p})^{-1}
\Del(X_H^{\p\p},X_G)
\wh{i}^H(X_H^\p,X_H^{\p\p}),
\end{align*}
where 
\begin{itemize}
\item $\CO_{X_G^\p}\in \Gam(\Fg_\reg(\BR))$, i.e., $X_G^\p\in \Fg_\reg(\BR)$ modulo $G(\BR)$-conjugacy;

\item $\CO_{X_H^\p}\in \Fs_H^{-1}(\CO_{X_H}^\st)$, i.e., $X_H^\p$ runs over the $H(\BR)$-conjugacy classes in the stable conjugacy class of $X_H$; 

\item $X^{\p\p}_H\in \Fh_{\Gstarreg}(\BR)$ modulo $H(\BR)$-conjugacy;

\item $w(X^{\p\p}_H) = |\Fs_H^{-1}(\CO_{X^{\p\p}_H})|$, the number of $H(\BR)$-conjugacy classes within the stable conjugacy class of $X^{\p\p}_H$;

\item $\gam_\psi(\Fg)$ and $\gam_\psi(\Fh)$ are the Weil constants associated with $(\Fg(\BR),\langle \cdot,\cdot\rangle_\Fg)$ and $(\Fh(\BR),\langle \cdot,\cdot\rangle_\Fh)$, respectively $($\cite[\S 1]{jlgl2}$)$. Here the non-degenerate conjugation-invariant symmetric bilinear form $\langle \cdot,\cdot\rangle_\Fh$ on $\Fh(\BR)$ is induced from $\langle \cdot,\cdot \rangle_\Fg$ as in \cite[VIII.~6]{MR1344131}.
\end{itemize}
Then for any $X_H\in \Fh_{\Gstarreg}(\BR)$ that is fixed, the following identity holds for any $X_G\in \Fg_\reg(\BR)$:
$$
\RD_{G,H}(X_H,X_G) = \wt{\RD}_{G,H}(X_H,X_G).
$$
\end{thm}

\subsection{Elliptic situation}\label{subsec:elliptic_situation}

Thanks to the discussion in \cite[\S3--\S6]{waldtransfert}, whose arguments remain valid over the real field 
(provided that \cite[Lemme~6.6]{waldtransfert} is replaced by \cite[(3.4.4)]{beuzart2020local}), 
the proof of Theorem \ref{thm:Waldspurger-Id} can be reduced to the following situation:
\begin{num}
\item\label{num:reduce_to_elliptic} 
The group $G$ is semisimple and simply connected (\cite[\S4]{waldtransfert}), 
and $X_H$ is $G^*$-elliptic; that is, $X_H$ is associated with a diagram 
$\FD(X_H,X_{G^*};\BR)$ in which $X_{G^*}$ is elliptic in $\Fg^*_\reg(\BR)$. 
By \cite[Lemme~6.1]{waldtransfert}, this is equivalent to the following:
\begin{itemize}
\item $(H,s,\xi)$ is an elliptic endoscopic datum, i.e., 
$\xi\big([Z_{\h{H}}^\Gam]^\circ\big)\subset Z_{\h{G}}$;
\item $X_H\in \Fh(\BR)_\el$.
\end{itemize}
\end{num}

In what follows, under the reduction \eqref{num:reduce_to_elliptic}, 
we establish Theorem \ref{thm:Waldspurger-Id} in the case where 
$X_G$ is also elliptic in $\Fg_\reg(\BR)$.

\begin{pro}\label{pro:Waldspurger-Id:elliptic}
Under \eqref{num:reduce_to_elliptic}, 
Theorem \ref{thm:Waldspurger-Id} holds when $X_G$ is elliptic in $\Fg_\reg(\BR)$.
\end{pro}
\begin{proof}

To prove the proposition, we first recall the following theorem of Rossmann.

\begin{thm}[\cite{MR0508985}, Corollary~(15), p.~217]\label{thm:rossmann_formula}
Let $\Ft^e(\BR)$ be an elliptic Cartan subalgebra of $\Fg(\BR)$. 
For any $X_G\in \Ft^e_\reg(\BR):=\Ft^e(\BR)\cap \Fg_\reg(\BR)$, the following hold:
\begin{enumerate}
\item $\wh{j}^G(X_G,\cdot)$ is locally integrable and is an $\RI(\Fg)$-eigendistribution on $\Fg(\BR)$;

\item Fix a positive root system $\Sig^+(\Ft^e)\subset \Sig(\Ft^e)$. 
For $X_G'\in \Ft^e_\reg(\BR)$, 
\begin{align*}
\wh{i}^G(X_G,X_G')
&= \RD^G(X_G')^{1/2}\,\wh{j}^G(X_G,X_G') 
\\ 
&= 
(-i)^{\frac{\dim(\Fg/\Ft^e)}{2}}
(-1)^{\frac{\dim(\Fg/\Fk)}{2}}
\frac{\RD^G(X_G)^{1/2}}{\pi_{\Sig^+(\Ft^e)}(X_G)}
\frac{\RD^G(X_G')^{1/2}}{\pi_{\Sig^+(\Ft^e)}(X_G')}
\\
&\quad 
\Big(
\sum_{w\in W^G(\Ft^e;\BR)}
\eps_G(w)\,
e^{\,i\langle w\cdot X_G,\,X_G'\rangle_\Fg}\,
\Big)
.
\end{align*}
\end{enumerate}
Here:
\begin{itemize}
\item $\Fk$ is the $+1$-eigenspace of the Cartan involution on $\Fg$;

\item $\pi_{\Sig^+(\Ft^e)}(X_G)=\prod_{\alp\in \Sig^+(\Ft^e)} \alp(X_G)$, so in particular $\RD^G(X_G)^{1/2} = |\pi_{\Sig^+(\Ft^e)}(X_G)|$.
Moreover, the product
\[
\frac{\RD^G(X_G)^{1/2}}{\pi_{\Sig^+(\Ft^e)}(X_G)}
\frac{\RD^G(X_G')^{1/2}}{\pi_{\Sig^+(\Ft^e)}(X_G')}
\]
is independent of the choice of the positive root system $\Sig^+(\Ft^e)$;

\item $W^G(\Ft^e;\BR)$ is the real Weyl group of $\Ft^e(\BR)$ in $\Fg(\BR)$, and 
$\eps_G(w)=\det\bigl(w:\Ft^e\to \Ft^e\bigr)$ is its signature.
\end{itemize}
\end{thm}

\begin{rmk}\label{rmk:roosmann_formula}
Compared to the original statement in \cite[Corollary~(15), p.~217]{MR0508985}, 
our formula misses an additional factor $|W(\Ft^e;\BR)|^{-1}$. 
This normalization arises from the choice of measure explained 
in the last line of \cite[p.~218]{MR0508985}. 
\end{rmk}

We also need the following lemma.

\begin{lem}\label{lem:counting}
We use the notation from Definition \ref{defin:Gstarreg}. Fix a pair $(X_H, X_G) \in \Fh_{\Gstarreg}(\BR) \times \Fg_\reg(\BR)$ associated to a diagram $\FD(X_H, X_G; \BR)$, where $X_H$ is $G^*$-elliptic and $X_G$ is elliptic in $\Fg_\reg(\BR)$. Then the following statements hold:
\begin{enumerate}
\item 
\[
\Fs_G^{-1}(\CO^\st_{X_G}) = 
\bigsqcup_{w \in W^G/W^G(\Ft^e; \BR)} \CO_{w \cdot X_G}.
\]
Here we recall:
\begin{itemize}
    \item 
    $\CO_{X_G}$ denotes the $G(\BR)$-conjugacy class of $X_G$ in $\Gam(\Fg_\reg(\BR)) = \Fg_\reg(\BR)/G(\BR)$, and $\CO^\st_{X_G}$ denotes the stable $G(\BR)$-conjugacy class of $X_G$ in $\Sig(\Fg_\reg(\BR)) = \Fg_\reg(\BR)/G(\BC)$;

    \item
    $\Fs_G: \Gam(\Fg_\reg(\BR)) \to \Sig(\Fg_\reg(\BR))$ is the natural projection;

    \item 
    $W^G$ is the absolute Weyl group of $\Ft^e(\BC)$ in $\Fg(\BC)$.
\end{itemize}

\item 
\begin{align*}
&\left\{X_G'\in \Fg_\reg(\BR) \mid \text{there exists a diagram } \FD(X_H, X_G'; \BR) \right\} / G(\BR) \\
=& \bigsqcup_{w \in W^G/W^G(\Ft^e; \BR)} \CO_{w \cdot X_G} 
= \Fs_G^{-1}(\CO^\st_{X_G}).
\end{align*}

\item 
\begin{align*}
&\left\{X_H'\in \Fh_{\Gstarreg}(\BR) \mid \text{there exists a diagram } \FD(X_H', X_G; \BR) \right\} / H(\BR) \\
=& \bigsqcup_{w \in W^G / W^H(\Ft_H^e; \BR)} \CO_{(\CI^{X_G}_{X_H})^{-1}(w \cdot X_G)}.
\end{align*}
Here $\CI^{X_G}_{X_H} : \Ft_{X_H} \xrightarrow{\sim} \Ft_{X_G}$ is the isomorphism over $\BR$ appearing in Definition \ref{defin:Gstarreg}. Since $X_H$ is $G^*$-elliptic and $X_G$ is elliptic, we may, up to $H(\BR)$- and $G(\BR)$-conjugacy respectively, identify $\Ft_{X_H}(\BR)$ and $\Ft_{X_G}(\BR)$ with fixed elliptic Cartan subalgebras $\Ft^e_H(\BR) \subset \Fh(\BR)$ and $\Ft^e(\BR) \subset \Fg(\BR)$, respectively. We also view $W^H(\Ft_H^e; \BR)$, the real Weyl group of $H(\BR)$, as a subgroup of $W^H$ and hence as a subgroup of $W^G$.
\end{enumerate}
\end{lem}
\begin{proof}
(1) The first statement follows from \cite[Corollary~2.5]{shelsteadinner}. In particular, for any $w \in W^G$, we have $w \cdot X_G \in \Ft^e(\BR) \cap \Fg_\reg(\BR)$.

(2) For the second statement, by Part (1) just established, it suffices to prove the following assertion:
\begin{num}
\item There exists a diagram $\FD(X_H, X_G'; \BR)$ if and only if $X_G'$ and $X_G$ are stably conjugate.
\end{num}

Let $\FD(X_H, X_G; \BR)$ be the diagram as in Definition \ref{defin:Gstarreg}. Since $X_H$ is $G^*$-elliptic, the associated $X_G$ is elliptic in $\Fg_\reg(\BR)$. By \cite[p.~208]{MR0508985}, up to $G(\BR)$-conjugation, there exists a unique elliptic Cartan subalgebra $\Ft^e(\BR)$ in $G(\BR)$.

Suppose that $X_G$ and $X_G'$ are stably conjugate. Then $X_G'$ is also elliptic. Up to $G(\BR)$-conjugation, we may assume both $X_G$ and $X_G'$ lie in $\Ft^e(\BR) \cap \Fg_\reg(\BR)$. Since $G = G_\sc$ is simply connected (by \eqref{num:reduce_to_elliptic}), there exists $g' \in G$ such that $\Ad(g')(X_G) = X_G'$. Moreover, by \cite[Proposition~2.2]{shelsteadinner}, the automorphism $\Ad(g'):\Ft^e \to \Ft^e$ is defined over $\BR$. It follows that there exists a diagram $\FD(X_H, X_G'; \BR)$ with the following data:
\begin{align*}
\begin{cases}
h \in H, \quad g^* \in G^*_\sc, \quad  g'g \in G_\sc, \\[0.5em]
\xymatrix{
T_{X_H} \ar[r]^{\Ad_h}
& \BT_H \ar[r]^{\eta} 
& \BT \ar[r]^{\Ad_{g^*}} 
& T_{X_{G^*}} 
\ar[rrr]^{\Ad_{g} \circ \vphi^{-1} = \CI^{X_G}_{X_{G^*}}} &&& 
T_{X_G} \ar[r]^{\Ad_{g'}} & T_{X_G'},
}
\end{cases}
\end{align*}
which satisfies the compatibility conditions in Definition \ref{defin:Gstarreg}.

Conversely, suppose there exists a diagram $\FD(X_H, X_G'; \BR)$ with the following data:
\begin{align*}
\begin{cases}
\tilde{h} \in H, \quad \tilde{g}^* \in G^*_\sc, \quad \tilde{g} \in G_\sc, \\[0.5em]
\xymatrix{
T_{X_H} \ar[r]^{\Ad_{\tilde{h}}}
& \BT_H \ar[r]^{\eta} 
& \BT \ar[r]^{\Ad_{\tilde{g}^*}} 
& T_{X_{G^*}'} 
\ar[rrr]^{\Ad_{\tilde{g}} \circ \vphi^{-1} = \CI^{X_G'}_{X_{G^*}'}} &&& T_{X_G'}.
}
\end{cases}
\end{align*}
We claim that $X_G$ and $X_G'$ are stably conjugate. Indeed, the automorphism $\Ad_{\tilde{h} h^{-1}} : \BT_H \to \BT_H$ corresponds to an element of the absolute Weyl group $W^H$. By Remark \ref{rmk:admissible_embed_Weylequivariant}, since $\eta$ is equivariant with respect to the inclusion $W^H \hookrightarrow W^G$, the induced automorphism 
\[
\eta \circ \Ad_{\tilde{h} h^{-1}} \circ \eta^{-1} : \BT \to \BT
\]
corresponds to an element $w^G \in W^G$. Similarly, the composition 
\[
\Ad_{\tilde{g}^*} \circ w^G \circ \Ad_{(g^*)^{-1}} : T_{X_{G^*}} \to T_{X_{G^*}'}
\]
is realized by conjugation by an element $r^* \in G^*$. Therefore, we obtain the following commutative diagram:
\[
\xymatrix{
T_{X_{G^*}'} \ar[rrr]^{\Ad_{\tilde{g}} \circ \vphi^{-1}} &&& T_{X_G'} \\ 
T_{X_{G^*}} \ar[u]^{\Ad_{r^*}} \ar[rrr]^{\Ad_g \circ \vphi^{-1}} &&& T_{X_G}.
}
\]
Passing to Lie algebras, the resulting composition sending $X_G$ to $X_G'$ is:
\[
\Ad_{\tilde{g}} \circ \vphi^{-1} \circ \Ad_{r^*} \circ \left( \Ad_g \circ \vphi^{-1} \right)^{-1}
= \Ad_{\tilde{g}} \circ \Ad_{\vphi^{-1}(r^*)} \circ \Ad_{g^{-1}}
= \Ad_{\tilde{g} \vphi^{-1}(r^*) g^{-1}}.
\]
Thus, $X_G$ and $X_G'$ are stably conjugate, completing the proof of the second statement.

(3) For the third statement, note that whenever a diagram $\FD(X_H', X_G; \BR)$ exists, the element $X_H'$ must be $G^*$-elliptic. In particular, we may assume that both $X_H$ and $X_H'$ lie in the same elliptic Cartan subalgebra $\Ft_H^e$ of $H(\BR)$.

We first show that if there exists a diagram $\FD(X_H', X_G; \BR)$, then there exists $w \in W^G$ such that
\[
\CO_{X_H'} = \CO_{(\CI^{X_G}_{X_H})^{-1}(w \cdot X_G)}.
\]
To see this, consider the diagrams $\FD(X_H, X_G; \BR)$ and $\FD(X_H', X_G; \BR)$, with the following data (for convenience we work with the Lie algebras):
\begin{align*}
\begin{cases}
h \in H, \quad g^* \in G^*_\sc, \quad g \in G_\sc, \\
\xymatrix{
X_H \in \Ft_H^e = \Ft_{X_H} \ar[r]^{\Ad_h}
& \Ft_H \ar[r]^{\eta} & \Ft \ar[r]^{\Ad_{g^*}} &
\Ft_{X_{G^*}} \ar[rrr]^{\Ad_g \circ \varphi^{-1} = \CI^{X_G}_{X_{G^*}}} &&& 
\Ft_{X_G} = \Ft^e \ni X_G
}
\end{cases}
\end{align*}
\begin{align*}
\begin{cases}
\tilde{h} \in H, \quad \tilde{g}^* \in G^*_\sc, \quad \tilde{g} \in G_\sc, \\
\xymatrix{
X_H' \in \Ft_H^e = \Ft_{X_H'} \ar[r]^{\Ad_{\tilde{h}}}
& \Ft_H \ar[r]^{\eta} & \Ft \ar[r]^{\Ad_{\tilde{g}^*}} &
\Ft_{X_{G^*}} \ar[rrr]^{\Ad_{\tilde{g}} \circ \varphi^{-1} = \CI^{X_G}_{X_{G^*}}} &&& 
\Ft_{X_G} = \Ft^e \ni X_G
}
\end{cases}
\end{align*}
To show that $\CO_{X_H'} = \CO_{(\CI^{X_G}_{X_H})^{-1}(w \cdot X_G)}$ for some $w \in W^G$, it suffices to establish the following:
\begin{num}
\item\label{num:part(3)}
The element
\[
\CI^{X_G}_{X_H}(X_H') = \Ad_g \circ \varphi^{-1} \circ \Ad_{g^*} \circ \eta \circ \Ad_h(X_H')
\]
is stably conjugate to $X_G$. Equivalently, there exists $w \in W^G$ such that
\[
\Ad_g \circ \varphi^{-1} \circ \Ad_{g^*} \circ \eta \circ \Ad_h(X_H') = w \cdot X_G.
\]
\end{num}
By definition, we have:
\begin{align}\label{eq:part(3):1}
\Ad_g \circ \varphi^{-1} \circ \Ad_{g^*} \circ \eta \circ \Ad_h(X_H') = 
\Ad_g \circ \varphi^{-1} \circ \Ad_{g^*} \circ \eta \circ 
\Ad_{h \tilde{h}^{-1}} \circ \Ad_{\tilde{h}}(X_H').
\end{align}
Since $\Ad_{h \tilde{h}^{-1}}$ normalizes $\Ft_H^e$, it corresponds to an element $w^H \in W^H$. Because $\eta$ is equivariant with respect to the inclusion $W^H \hookrightarrow W^G$, the conjugation $\eta \circ \Ad_{h \tilde{h}^{-1}} \circ \eta^{-1}$ corresponds to an element $w^G \in W^G$. Hence, we may rewrite \eqref{eq:part(3):1} as
\[
\Ad_g \circ \varphi^{-1} \circ \Ad_{g^*} \circ w^G \circ \eta \circ \Ad_{\tilde{h}}(X_H').
\]
Next, using the identity
\[
\Ad_{g^*} \circ w^G = w^G \circ \Ad_{w^G \cdot g^*} = 
w^G \circ \Ad_{(w^G \cdot g^*) (\tilde{g}^*)^{-1}} \circ \Ad_{\tilde{g}^*},
\]
and noting that the composition 
\[
w^G \circ \Ad_{(w^G \cdot g^*) (\tilde{g}^*)^{-1}}
\]
is again represented by an element of $W^G$, which we denote abusively by $w^G$, the entire composition becomes
\[
\Ad_g \circ \varphi^{-1} \circ w^G \circ \Ad_{\tilde{g}^*} \circ \eta \circ \Ad_{\tilde{h}}(X_H').
\]
Arguing as before, we conclude that this expression equals
\[
w \cdot (\Ad_{\tilde{g}} \circ \varphi^{-1} \circ \Ad_{\tilde{g}^*} \circ \eta \circ \Ad_{\tilde{h}})(X_H') = w \cdot X_G
\]
for some $w \in W^G$, completing the proof of \eqref{num:part(3)}.

To complete the proof of Part (3), it remains to show the following:
\begin{num}
\item\label{num:part(3):2}
For $w \in W^G$, we have
\[
\CO_{(\CI^{X_G}_{X_H})^{-1}(w \cdot X_G)} = \CO_{X_H}
\quad \text{if and only if} \quad
w \in W^H(\Ft_H^e; \BR).
\]
\end{num}
By definition, 
\[
\CO_{(\CI^{X_G}_{X_H})^{-1}(w \cdot X_G)} = \CO_{X_H}
\]
if and only if $(\CI^{X_G}_{X_H})^{-1}(w \cdot X_G)$ is $H(\BR)$-conjugate to $X_H$. Equivalently, this occurs if and only if they are conjugate under the real Weyl group $W^H(\Ft_H^e; \BR)$ of $H(\BR)$.

Since the isomorphism $\CI^{X_G}_{X_H} : \Ft_H^e \to \Ft^e$ is equivariant with respect to Weyl group actions (via the equivariance of $\eta$), we conclude that 
$(\CI^{X_G}_{X_H})^{-1}(w \cdot X_G)$ is $W^H(\Ft_H^e; \BR)$-conjugate to $X_H$ if and only if there exists $w_0 \in W^H(\Ft_H^e; \BR)$ such that
\begin{equation}\label{eq:part(3):last}
w_0 \cdot w \cdot X_G = X_G.
\end{equation}
Since we may view $W^H(\Ft_H^e; \BR) \subset W^H \subset W^G$, and $X_G$ is regular semi-simple, the stabilizer of $X_G$ in $W^G$ is trivial. Thus, \eqref{eq:part(3):last} holds if and only if $w \in W^H(\Ft_H^e; \BR)$. This completes the proof.
\end{proof}

With the aid of Theorem~\ref{thm:rossmann_formula} and Lemma~\ref{lem:counting}, we obtain the following explicit formulae for $\RD_{G,H}(X_H, X_G)$ and $\wt{\RD}_{G,H}(X_H, X_G)$.

\begin{lem}\label{lem:explicit_formulae}
Under \eqref{num:reduce_to_elliptic}, when $X_H$ is $G^*$-elliptic and $X_G$ is elliptic in $\Fg_\reg(\BR)$, 
\begin{align*}
\RD_{G,H}(X_H,X_G) =& 
\gam_\psi(\Fg) (-i)^{\frac{\dim\Fg/\Ft^e}{2}}(-1)^{\frac{\dim\Fg/\Fk}{2}}
\frac{\RD^G(X_G)^{1/2}}{
\pi_{\Sig^+(\Ft^e)}(X_G)}
\\ 
&
\bigg(
\sum_{w\in W^G}
\Del(X_H,w\cdot X^H_G) 
\frac{\RD^G(w\cdot X^H_G)^{1/2}}{
\pi_{\Sig^+(\Ft^e)}(w\cdot X^H_G)
}
e^{i\langle w\cdot X^H_G,X_G\rangle_\Fg}
\bigg).
\\ 
\wt{\RD}_{G,H}(X_H,X_G) = &
\gam_\psi(\Fh)
(-i)^{\frac{\dim \Fh/\Ft_H^e}{2}}(-1)^{\frac{\dim \Fh/\Fk_\Fh}{2}}
\frac{\RD^H(X_H)^{1/2}}{
\pi_{\Sig^+_H(\Ft_H^e)}(X_H)
}     
\bigg(
\sum_{
\substack{ 
w \in W^G
}}
\Del\big((\CI^e)^{-1}(w\cdot X_G),X_G\big)
          \nonumber
\\ 
&
\frac{\RD^H\big((\CI^e)^{-1}(w\cdot X_G)\big)^{1/2}}{
\pi_{\Sig^+_H(\Ft_H^e)}\big((\CI^e)^{-1}(w\cdot X_G) \big)}
e^{i\langle X_H,(\CI^e)^{-1}(w\cdot X_G)\rangle_\Fh}
\bigg).
\end{align*}
Here $X^H_G \in \Fg_\reg(\BR)\cap \Ft^e(\BR)$ is fixed such that $\Del(X_H,X^H_G)\neq 0$, i.e., there exists a diagram $\FD(X_H,X^H_G;\BR)$. For abbreviation we write $\CI^e = \CI^{X^H_G}_{X_H}:\Ft_{X_H}\simeq \Ft^e_H\simeq \Ft^e\simeq \Ft_{X^H_G}$. In particular $\CI^e(X_H) = X^H_G$.
\end{lem}
\begin{proof}
We first compute \( \RD_{G,H}(X_H, X_G) \). By definition,
\begin{align*}
\RD_{G,H}(X_H, X_G) = 
\gamma_\psi(\Fg)
\sum_{\CO_{X_G'} \in \Gam(\Fg_\reg(\BR))}
\Delta(X_H, X_G')
\widehat{i}^G(X_G', X_G).
\end{align*}
Since \( X_H \) is \( G^* \)-elliptic, the term \( \Delta(X_H, X_G') \) is nonzero only if there exists a diagram \( \FD(X_H, X_G'; \BR) \), which in particular implies that \( X_G' \) is elliptic in \( \Fg_\reg(\BR) \). Moreover, once we fix \( X_G^H \in \Fg_\reg(\BR)\cap \Ft^e(\BR) \) such that \( \Delta(X_H, X_G^H) \neq 0 \), Part (1) of Lemma~\ref{lem:counting} implies:
\begin{align*}
\RD_{G,H}(X_H, X_G) 
&= \gamma_\psi(\Fg) \sum_{w \in W^G / W^G(\Ft^e; \BR)}
\Delta(X_H, w \cdot X_G^H) 
\widehat{i}^G(w \cdot X_G^H, X_G) \\
&= \frac{\gamma_\psi(\Fg)}{|W^G(\Ft^e; \BR)|}
\sum_{w \in W^G}
\Delta(X_H, w \cdot X_G^H) 
\widehat{i}^G(w \cdot X_G^H, X_G).
\end{align*}

Applying Theorem~\ref{thm:rossmann_formula}, we obtain:
\begin{align*}
\RD_{G,H}(X_H, X_G) ={}&
\frac{\gamma_\psi(\Fg)}{|W^G(\Ft^e; \BR)|}
\bigg(
\sum_{w \in W^G}
\Delta(X_H, w \cdot X_G^H)
(-i)^{\frac{\dim \Fg / \Ft^e}{2}}
(-1)^{\frac{\dim \Fg / \Fk}{2}} \\
&
\frac{\RD^G(w \cdot X_G^H)^{1/2}}{
\pi_{\Sigma^+(\Ft^e)}(w \cdot X_G^H)}
\frac{\RD^G(X_G)^{1/2}}{
\pi_{\Sigma^+(\Ft^e)}(X_G)}
\Big(
\sum_{w' \in W^G(\Ft^e; \BR)} \epsilon_G(w') \cdot
e^{i \langle w' w \cdot X_G^H, X_G \rangle_\Fg}
\Big)\bigg).
\end{align*}

Now, change variables via \( w \mapsto (w')^{-1} w \), and use the identity:
\[
\pi_{\Sigma^+(\Ft^e)}((w')^{-1} w \cdot X_G^H) = \epsilon(w')\pi_{\Sigma^+(\Ft^e)}(w \cdot X_G^H),
\]
to obtain:
\begin{align*}
\RD_{G,H}(X_H, X_G) ={}&
\gamma_\psi(\Fg) 
(-i)^{\frac{\dim \Fg / \Ft^e}{2}} 
(-1)^{\frac{\dim \Fg / \Fk}{2}} 
\frac{\RD^G(X_G)^{1/2}}{\pi_{\Sigma^+(\Ft^e)}(X_G)} \\
&
\bigg(
\sum_{w \in W^G}
\Delta(X_H, w \cdot X_G^H) 
\frac{\RD^G(w \cdot X_G^H)^{1/2}}{\pi_{\Sigma^+(\Ft^e)}(w \cdot X_G^H)}
e^{i \langle w \cdot X_G^H, X_G \rangle_\Fg}
\bigg)
.
\end{align*}

We now compute \( \wt{\RD}_{G,H}(X_H, X_G) \). By definition,
\[
\wt{\RD}_{G,H}(X_H, X_G) = 
\gamma_\psi(\Fh) 
\sum_{
\substack{
\CO_{X_H'} \in \Fs_H^{-1}(\CO_{X_H}^\st) \\
\CO_{X_H''} \in \Gam(\Fh_{\Gstarreg}(\BR))
}
}
w(X_H'')^{-1} 
\Delta(X_H'', X_G) 
\widehat{i}^H(X_H', X_H'').
\]
We make the following observations:
\begin{itemize}
    \item Since \( X_G \) is elliptic, the term \( \Delta(X_H'', X_G) \) is nonzero only when \( X_H'' \) is \( G^* \)-elliptic, and hence elliptic. By Part (1) of Lemma~\ref{lem:counting},
    \[
    w(X_H'') = |W^H / W^H(\Ft_H^e; \BR)|.
    \]
    
    \item Without loss of generality, assume \( X_H \in \Ft_H^e(\BR) \cap \Fh_{\Gstarreg}(\BR) \) and \( \Delta(X_H, X_G) \neq 0 \). Then by Part (3) of Lemma~\ref{lem:counting}, up to \( H(\BR) \)-conjugation, the set of \( X_H' \in \Fh_{\Gstarreg}(\BR) \) such that \( \Delta(X_H', X_G) \neq 0 \) is given by
    \[
    \left\{
    X_H' \in \Fh_{\Gstarreg}(\BR) \mid 
    \Delta(X_H', X_G) \neq 0
    \right\} / H(\BR) = 
    \bigsqcup_{w \in W^G / W^H(\Ft_H^e; \BR)}
    \CO_{(\CI^e)^{-1}(w \cdot X_G)}.
    \]
    
    \item By Part (1) of Lemma~\ref{lem:counting},
    \[
    \Fs_H^{-1}(\CO_{X_H}^\st) = 
    \bigsqcup_{w \in W^H / W^H(\Ft_H^e; \BR)}
    \CO_{w \cdot X_H}.
    \]
\end{itemize}
Hence,
\begin{align*}
\wt{\RD}_{G,H}(X_H, X_G) = 
\gamma_\psi(\Fh)
\sum_{
\substack{
w' \in W^H \\
w \in W^G
}
}
\frac{
\Delta\big((\CI^e)^{-1}(w \cdot X_G), X_G\big)
\widehat{i}^H\big(w' \cdot X_H, (\CI^e)^{-1}(w \cdot X_G)\big)
}{
|W^H(\Ft_H^e; \BR)|  |W^H|
}.
\end{align*}
Applying Theorem~\ref{thm:rossmann_formula}, we obtain:
\begin{align*}
\wt{\RD}_{G,H}(X_H, X_G) ={}&
\gamma_\psi(\Fh)
\frac{(-i)^{\frac{\dim \Fh / \Ft_H^e}{2}} (-1)^{\frac{\dim \Fh / \Fk_\Fh}{2}}}{|W^H(\Ft_H^e; \BR)| |W^H|}
\bigg(
\sum_{
\substack{
w' \in W^H \\
w \in W^G \\
w'' \in W^H(\Ft_H^e; \BR)
}
}
\Delta\big((\CI^e)^{-1}(w \cdot X_G), X_G\big)
\\
&
\frac{\RD^H(w' \cdot X_H)^{1/2}}{\pi_{\Sigma_H^+(\Ft_H^e)}(w' \cdot X_H)}
\frac{\RD^H((\CI^e)^{-1}(w \cdot X_G))^{1/2}}{\pi_{\Sigma_H^+(\Ft_H^e)}((\CI^e)^{-1}(w \cdot X_G))}
\epsilon_H(w'')
e^{i \langle w'' w' \cdot X_H, (\CI^e)^{-1}(w \cdot X_G) \rangle_\Fh}
\bigg)
.
\end{align*}
By definition, 
$$
\frac{\RD^H(w^\p\cdot X_H)^{1/2}}{
\pi_{\Sig^+_H(\Ft_H^e)}(w^\p\cdot X_H)
}=
\eps_H(w^{\p\p})
\frac{\RD^H(w^{\p\p}w^\p\cdot X_H)^{1/2}}{
\pi_{\Sig^+_H(\Ft_H^e)}
(w^{\p\p}w^\p\cdot X_H)
}.
$$
Hence, after changing variable $w^\p\mapsto (w^{\p\p})^{-1}w^\p$, we get 
\begin{align*}
\wt{\RD}_{G,H}(X_H,X_G) = &
\gam_\psi(\Fh)
\frac{(-i)^{\frac{\dim \Fh/\Ft_H^e}{2}}(-1)^{\frac{\dim \Fh/\Fk_\Fh}{2}}}{|W^H|} 
\bigg(\sum_{
\substack{
w^\p\in W^H \\ 
w\in W^G
}
}
\Del\big((\CI^e)^{-1}(w\cdot X_G),X_G\big)
\\ 
&
\frac{\RD^H(w^\p \cdot X_H)^{1/2}}{\pi_{\Sig^+_H(\Ft_H^e)}(w^\p \cdot X_H)}
\frac{\RD^H\big((\CI^e)^{-1}(w\cdot X_G)\big)^{1/2}}{\pi_{\Sig^+_H(\Ft_H^e)}\big((\CI^e)^{-1}(w\cdot X_G)\big)} 
e^{i\langle w^\p\cdot X_H,(\CI^e)^{-1}(w\cdot X_G)\rangle_\Fh}
\bigg)
.
\end{align*}
Changing variable $w\mapsto w^{\p}w$, using the fact that $\CI^e$ is equivariant with respect to $W^H\hookrightarrow W^G$, we get
$$
\Del\big((\CI^e)^{-1}(w^\p w\cdot X_G),X_G\big)
=\Del\big(w^{\p}\cdot (\CI^e)^{-1}(w\cdot X_G),X_G\big)
=
\Del\big((\CI^e)^{-1}(w\cdot X_G),X_G\big).
$$
Similarly, since the bilinear pairing $\langle \cdot,\cdot \rangle_\Fh$ is conjugation invariant, we get
$$
\langle w^\p\cdot X_H,(\CI^e)^{-1}(w^{\p}w\cdot X_G)\rangle_\Fh
=
\langle w^\p\cdot X_H,w^{\p}\cdot (\CI^e)^{-1}(w\cdot X_G)\rangle_\Fh
=
\langle X_H,(\CI^e)^{-1}(w\cdot X_G)\rangle_\Fh.
$$
Hence we get
\begin{align}\label{eq:eq:elliptic:DtildGH}
\wt{\RD}_{G,H}(X_H,X_G) = &
\gam_\psi(\Fh)
(-i)^{\frac{\dim \Fh/\Ft_H^e}{2}}(-1)^{\frac{\dim \Fh/\Fk_\Fh}{2}}
\frac{\RD^H(X_H)^{1/2}}{
\pi_{\Sig^+_H(\Ft_H^e)}(X_H)
}     
\bigg(
\sum_{
\substack{ 
w \in W^G
}}
\Del\big((\CI^e)^{-1}(w\cdot X_G),X_G\big)
          \nonumber
\\ 
&
\frac{\RD^H\big((\CI^e)^{-1}(w\cdot X_G)\big)^{1/2}}{
\pi_{\Sig^+_H(\Ft_H^e)}\big((\CI^e)^{-1}(w\cdot X_G) \big)}
e^{i\langle X_H,(\CI^e)^{-1}(w\cdot X_G)\rangle_\Fh}
\bigg).
\end{align}
It follows that we finish the proof of the lemma.
\end{proof}

Now we are ready to prove the proposition. Based on the lemma we just proved, we have shown the following fact:

\begin{num}
\item\label{num:explicit:DandtildeD}

With the notation as above we have the following explicit formulae:
\begin{align*}
\RD_{G,H}(X_H,X_G) =& \gam_\psi(\Fg)
(-i)^{\frac{\dim \Fg/\Ft^e}{2}}
(-1)^{\frac{\dim \Fg/\Fk}{2}}
\frac{\RD^G(X_G)^{1/2}}{\pi_{\Sig^+(\Ft^e)}(X_G)}   \nonumber
\\ 
&
\bigg(
\sum_{w\in W^G}
\Del(X_H,w\cdot X^H_G)
\frac{\RD^G(w\cdot X^H_G)^{1/2}}{\pi_{\Sig^+(\Ft^e)}(w\cdot X^H_{G})}
e^{i\langle w\cdot X^H_G,X_G\rangle_\Fg}
\bigg)
\\ 
\wt{\RD}_{G,H}(X_H,X_G) = &
\gam_\psi(\Fh)
(-i)^{\frac{\dim \Fh/\Ft_H^e}{2}}(-1)^{\frac{\dim \Fh/\Fk_\Fh}{2}}
\frac{\RD^H(X_H)^{1/2}}{
\pi_{\Sig^+_H(\Ft_H^e)}(X_H)
}     
\bigg(
\sum_{
\substack{ 
w \in W^G
}}
\Del\big((\CI^e)^{-1}(w\cdot X_G),X_G\big)
          \nonumber
\\ 
&
\frac{\RD^H\big((\CI^e)^{-1}(w\cdot X_G)\big)^{1/2}}{
\pi_{\Sig^+_H(\Ft_H^e)}\big((\CI^e)^{-1}(w\cdot X_G) \big)}
e^{i\langle X_H,(\CI^e)^{-1}(w\cdot X_G)\rangle_\Fh}
\bigg).
\end{align*}
\end{num}

After fixing a Killing form on $\Fg$ and $\Fh$, $\Fg/\Fk_G$ and $\Fh/\Fk_H$ are positive definite, while $\Fk_G$ and $\Fk_H$ are negative definite. Hence from \cite[\S 1]{jlgl2}, we have 
\begin{align*}
\gam_\psi(\Fg)/\gam_\psi(\Fk_G) &= i^{\frac{\dim\Fg/\Fk_G}{2}}, \quad 
\gam_\psi(\Fk_G) = (-i)^{\frac{\dim\Fk_G}{2}},
\\ 
\gam_\psi(\Fh)/\gam_\psi(\Fk_H) &= i^{\frac{\dim\Fh/\Fk_H}{2}}, \quad 
\gam_\psi(\Fk_H) = (-i)^{\frac{\dim\Fk_H}{2}}.
\end{align*}
It follows that 
\begin{align*}
\gam_\psi(\Fg)
(-i)^{\frac{\dim \Fg/\Ft^e}{2}}
(-1)^{\frac{\dim \Fg/\Fk}{2}}
=
\gam_\psi(\Fh)
(-i)^{\frac{\dim \Fh/\Ft_H^e}{2}}(-1)^{\frac{\dim \Fh/\Fk_\Fh}{2}}.
\end{align*}
Hence we are reduced to show that 
\begin{align}\label{eq:pro:equality:last_id}
&\frac{\RD^G(X_G)^{1/2}}{\pi_{\Sig^+(\Ft^e)}(X_G)}
\bigg(
\sum_{w\in W^G}
\Del(X_H,w\cdot X^H_G)
\frac{\RD^G(w\cdot X^H_G)^{1/2}}{\pi_{\Sig^+(\Ft^e)}(w\cdot X^H_{G})}
e^{i\langle w\cdot X^H_G,X_G\rangle_\Fg}
\bigg)                  \nonumber
\\ 
=&
\frac{\RD^H(X_H)^{1/2}}{
\pi_{\Sig^+_H(\Ft_H^e)}(X_H)
}     
\bigg(
\sum_{
\substack{ 
w \in W^G
}}
\Del\big((\CI^e)^{-1}(w\cdot X_G),X_G\big)
\frac{\RD^H\big((\CI^e)^{-1}(w\cdot X_G)\big)^{1/2}}{
\pi_{\Sig^+_H(\Ft_H^e)}\big((\CI^e)^{-1}(w\cdot X_G) \big)}
e^{i\langle X_H,(\CI^e)^{-1}(w\cdot X_G)\rangle_\Fh}
\bigg).
\end{align}
To prove \eqref{eq:pro:equality:last_id}, it suffices to show the following statement:
\begin{num}
\item\label{num:lastred:DGH=tildDGH} For any $w\in W^G$, 
\begin{align*}
&\frac{\RD^G(X_G)^{1/2}}{\pi_{\Sig_G^+(\Ft^e)}(X_G)}
\Del(X_H,w\cdot X^H_G)
\frac{\RD^G(w\cdot X^H_G)^{1/2}}{\pi_{\Sig_G^+(\Ft^e)}(w\cdot X^H_{G})}
e^{i\langle w\cdot X^H_G,X_G\rangle_\Fg}
\\ 
=&
\frac{\RD^H(X_H)^{1/2}}{
\pi_{\Sig^+_H(\Ft_H^e)}(X_H)
}     
\Del\big((\CI^e)^{-1}(w^{-1}\cdot X_G),X_G\big)
\frac{\RD^H\big((\CI^e)^{-1}(w^{-1}\cdot X_G)\big)^{1/2}}{
\pi_{\Sig^+_H(\Ft_H^e)}\big((\CI^e)^{-1}(w^{-1}\cdot X_G) \big)}
e^{i\langle X_H,(\CI^e)^{-1}(w^{-1}\cdot X_G)\rangle_\Fh}.
\end{align*}
\end{num}
Notice that both $(X_H,w\cdot X^H_G)$ and $\big(
(\CI^e)^{-1}(w^{-1}\cdot X_G)
,X_G\big)$ in $\Fh_{\Gstarreg}(\BR)\times \Fg_\reg(\BR)$ come from the following composition of isomorphisms
$$
\xymatrix{
X_H,(\CI^{e})^{-1}(w^{-1}\cdot X_G)\in \Ft^e_H=\Ft_{X_H} \ar[rr]^{\CI^e}_\simeq && \Ft_{X_G}= \Ft^e \ni X^H_G,w^{-1}\cdot X_G
\ar[d]_{\Ad_w} \\ 
&&
\Ft^e\ni
w\cdot X^H_G,X_G
}.
$$
As mentioned in Remark \ref{rmk:Del_I_dependence} and Remark \ref{rmk:Del_III_dependence}, by the definition of the transfer factor $\Del = \Del_{\RI}\cdot \Del_{\RI\RI}\cdot \Del_{\RI\RI\RI}$, both $\Del_\RI$ and $\Del_{\RI\RI\RI}$ depend only on the composition of the isomorphism $\Ad_w\circ \CI^e$ rather than the particular elements inside. Hence 
$$
\Del_{\RI}(X_H,w\cdot X^H_G)=
\Del_{\RI}\big(
(\CI^e)^{-1}(w^{-1}\cdot X_G),X_G
\big)\quad 
\Del_{\RI\RI\RI}(X_H,w\cdot X^H_G)=
\Del_{\RI\RI\RI}\big(
(\CI^e)^{-1}(w^{-1}\cdot X_G),X_G
\big)
$$
Moreover, by the definition of the Killing form on $\Fh$ from \cite[VIII.~6]{MR1344131}, we have
$$
\langle w\cdot X^H_G,X_G\rangle_\Fg = 
\langle w\cdot \CI^e(X_H),X_G\rangle_\Fg
=
\langle \CI^e(X_H),w^{-1}\cdot X_G\rangle_\Fg
=
\langle X_H,
(\CI^e)^{-1}(
w^{-1}\cdot X_G)\rangle_\Fh.
$$
Hence in order to establish \eqref{num:lastred:DGH=tildDGH}, it suffices to show the following identity
\begin{align}\label{eq:last_Del_II:1}
&\frac{\RD^G(X_G)^{1/2}}{\pi_{\Sig_G^+(\Ft^e)}(X_G)}
\Del_{\RI\RI}(X_H,w\cdot X^H_G)
\frac{\RD^G(w\cdot X^H_G)^{1/2}}{\pi_{\Sig_G^+(\Ft^e)}(w\cdot X^H_{G})} 
\\ 
=&
\frac{\RD^H(X_H)^{1/2}}{
\pi_{\Sig^+_H(\Ft_H^e)}(X_H)
}     
\Del_{\RI\RI}\big((\CI^e)^{-1}(w^{-1}\cdot X_G),X_G\big)
\frac{\RD^H\big((\CI^e)^{-1}(w^{-1}\cdot X_G)\big)^{1/2}}{
\pi_{\Sig^+_H(\Ft_H^e)}\big((\CI^e)^{-1}(w^{-1}\cdot X_G) \big)}.
\nonumber
\end{align}
But it follows from the definition of $\Del_{\RI\RI}$. Precisely, from Definition \ref{eq:defin:Del_II},
$$
\Del_{\RI\RI}(X_H,w\cdot X^H_G) = 
\prod 
\chi_\alp\bigg(
\frac{\alp(w\cdot X^H_G)}{a_\alp}
\bigg),\quad 
\Del_{\RI\RI}\big( (\CI^e)^{-1}(w^{-1}\cdot X_G),X_G\big) = 
\prod \chi_\alp
\bigg(
\frac{\alp(X_G)}{a_\alp}
\bigg),
$$
where the product is taken over a set of representatives of $\Gam$-orbits in $\Sig^{\sym}_{\text{hors }H}(T_{X_{G^*}}) = \Sig_G(\Ft^e)\bs \Sig_H(\Ft_H^e)$. Here $\Sig_H(\Ft^e_H)$ is identified as a subset of $\Sig_G(\Ft^e)$ via the isomorphism $\Ad_w\circ \CI^e:\Ft^e_H\simeq \Ft^e$. Since $\Ft^e$ is a fundamental Cartan subalgebra, we deduce that any $\alp\in \Sig_G(\Ft^e)$ are imaginary. Hence a set of representatives of $\Gam$-orbits in $\Sig_G(\Ft^e)\bs \Sig_H(\Ft_H^e)$ can be taken as $\Sig^+_G(\Ft^e)\bs \Sig^+_H(\Ft_H^e)$. Also since the ground field is the field of real numbers, $\chi_\alp = \sgn$ is the sign character. Hence 
\begin{equation}\label{eq:last_Del_II:2}
\frac{\Del_{\RI\RI}(X_H,w\cdot X^H_G)}{
\Del_{\RI\RI}\big((\CI^e)^{-1}(w^{-1}\cdot X_G),X_G\big)
}=
\prod_{\alp\in \Sig^+_G(\Ft^e)\bs \Sig^+_H(\Ft_H^e)}
\sgn
\bigg(
\frac{\alp(w\cdot X^H_G)}{\alp(X_G)}
\bigg).
\end{equation}
Finally, by the definition of $\Del_{\RI\RI}$, since 
$$
\Ad_w \circ \CI^e(X_H) = w\cdot X^H_G,\quad 
\Ad_w\circ \CI^e
\big(
(\CI^e)^{-1}(w^{-1}\cdot X_G)
\big) = X_G,
$$
we get that for any $\alp\in \Sig^+_{H}(\Ft^e_H)\hookrightarrow \Sig^+_G(\Ft^e)$
\begin{align}\label{eq:last_Del_II:3}
\alp(w\cdot X^H_G) = 
\alp
\big(
\Ad_w\circ \CI^e(X_H)
\big)
=
 \alp(X_H)
\quad 
\alp\big(
(\CI^e)^{-1}(w^{-1}\cdot X_G)\big) = 
\alp(X_G).
\end{align}
By combining \eqref{eq:last_Del_II:2} and \eqref{eq:last_Del_II:3}, we obtain the identity \eqref{eq:last_Del_II:1}, thereby completing the proof of the proposition.
\end{proof}

\subsection{Harish-Chandra uniqueness theorem}\label{subsec:harish_chandra_s_uniqueness_theorem}

In this subsection, we show that under the reduction \eqref{num:reduce_to_elliptic}, Proposition~\ref{pro:Waldspurger-Id:elliptic} suffices to deduce Theorem~\ref{thm:Waldspurger-Id} for any \( X_G' \in \Fg_\reg(\BR) \), with the aid of the Harish-Chandra uniqueness theorem, i.e., \( \RI(\Fg) \)-finite, invariant, tempered distributions on \( \Fg(\BR) \) of \textbf{elliptic type} are uniquely determined by their values on the elliptic locus. By showing that both
\[
\frac{\RD_{G,H}(X_H, \cdot)}{\RD^G(\cdot)^{1/2}} \quad \text{and} \quad \frac{\wt{\RD}_{G,H}(X_H, \cdot)}{\RD^G(\cdot)^{1/2}}
\]
satisfy the hypotheses of the Harish-Chandra uniqueness theorem, we reduce the proof of Theorem~\ref{thm:Waldspurger-Id} to Proposition~\ref{pro:Waldspurger-Id:elliptic}.

We begin by recalling the statement of the Harish-Chandra uniqueness theorem. Let \( G \) be a reductive algebraic group over \( \BR \) with Lie algebra \( \Fg \). Recall that \( \RI(\Fg) \) is the algebra of constant coefficient differential operators on \( \Fg(\BC) \) which acts naturally on the space of tempered distributions on \( \Fg(\BR) \). A tempered distribution on \( \Fg(\BR) \) is said to be \emph{invariant} if it is invariant under the conjugation action of \( G(\BR) \), and it is called an \emph{eigendistribution} if it is an eigenvector for the action of \( \RI(\Fg) \), with eigencharacter given by a \( G(\BR) \)-invariant polynomial.

The following lemma is due to Harish-Chandra (\cite[Part~I, §7.1]{MR0473111}).

\begin{lem}\label{lem:hatJeigen}
For any \( X \in \Fg_\reg(\BR) \), the tempered distribution \( \widehat{j}^G(X, \cdot) \), which is the Fourier transform of \( \RJ_G(X, \cdot) \), is a \( G(\BR) \)-invariant tempered eigendistribution. That is, there exists a homomorphism
\[
\chi^G_{[X]} : \RI(\Fg) \to \BC
\]
such that for all \( z \in \RI(\Fg) \),
\[
\partial(z)  \widehat{j}^G(X, \cdot) = 
\chi^G_{[X]}(z) \widehat{j}^G(X, \cdot).
\]
\end{lem}

\begin{rmk}\label{rmk:hatJeigen}
From \cite[Part~I, §7.1]{MR0473111}$)$, after identifying $\RI(\Fg)$ with the space of conjugation invariant polynomials on $\Fg(\BC)$ via a conjugation invariant non-degenerate symmetric bilinear form on $\Fg(\BC)$:
\begin{align*}
\RI(\Fg) &\longleftrightarrow \BC[\Fg(\BC)]
\\ 
z &\longleftrightarrow \h{z}
\end{align*}
we have
$$
\chi^G_{[X]}(z) = \h{z}(iX).
$$
By Hilbert's Nullstellensatz, up to conjugation $\chi^G_{[X]}$ is uniquely determined by $X$.
\end{rmk}

Following \cite[Theorem~13]{MR0473111} and \cite[p.118]{MR0473111}, we introduce the following definition.

\begin{defin}\label{defin:elliptic_type_distribution}
A tempered, invariant, \( \RI(\Fg) \)-finite distribution \( \mu \) on \( \Fg(\BR) \) is said to be of \emph{elliptic type} if there exists a finite set of elliptic regular semi-simple elements
\[
\CX = \{X_j\}_{j=1}^{m} \subset \Fg(\BR)_\el
\]
such that
\[
\partial(\mathcal{I}_\CX) \mu = 0,
\quad \text{where} \quad
\mathcal{I}_\CX := \bigcap_{j=1}^m \ker(\chi^G_{[X_j]}).
\]
\end{defin}

In particular, for any \( X \in \Fg(\BR)_\el \), the distribution \( \widehat{j}^G(X, \cdot) \) is of elliptic type.

The following result is due to Harish-Chandra (\cite[Theorem~13, p.~116]{MR0473111}).

\begin{thm}\label{thm:elliptic_unique}
Suppose that \( \Fg(\BR) \) admits an elliptic Cartan subalgebra \( \Ft^e(\BR) \). Let \( \mu \) be a tempered, invariant, \( \RI(\Fg) \)-finite distribution of elliptic type, and let \( \RF_\mu \) denote the analytic function on \( \Fg_\reg(\BR) \) associated to \( \mu \). Then
\[
\RF_\mu|_{\Ft^e(\BR)\cap \Fg_\reg(\BR)} = 0
\quad \Longrightarrow \quad
\mu = 0.
\]
In particular, any finite linear combination of tempered, invariant eigendistributions of elliptic type is uniquely determined by its restriction to the elliptic locus.
\end{thm}

To complete the proof of Theorem~\ref{thm:Waldspurger-Id} under the reduction 
\eqref{num:reduce_to_elliptic}, it remains, using 
Theorem~\ref{thm:elliptic_unique} together with 
Proposition~\ref{pro:Waldspurger-Id:elliptic}, to establish the following proposition.

\begin{pro}\label{pro:D&tildD_are_both_elliptic_type}
Under the assumption \eqref{num:reduce_to_elliptic}, if \( X_H \) is \( G^* \)-elliptic, then both
\[
\frac{\RD_{G,H}(X_H, \cdot)}{\RD^G(\cdot)^{1/2}}
\quad \text{and} \quad
\frac{\wt{\RD}_{G,H}(X_H, \cdot)}{\RD^G(\cdot)^{1/2}}
\]
are tempered, invariant, \( \RI(\Fg) \)-finite distributions of elliptic type.
\end{pro}
\begin{proof}
By definition,
\[
\frac{\RD_{G,H}(X_H, X_G)}{\RD^G(X_G)^{1/2}} =
\gamma_\psi(\Fg)
\sum_{\CO_{X_G'} \in \Gam(\Fg_\reg(\BR))}
\Delta(X_H, X_G') 
\widehat{j}^G(X_G', X_G).
\]
Since \( X_H \) is \( G^* \)-elliptic, we have \( X_G' \in \Fg(\BR)_\el \) whenever \( \Delta(X_H, X_G') \neq 0 \). It follows that
\[
\frac{\RD_{G,H}(X_H, \cdot)}{\RD^G(\cdot)^{1/2}}
\]
is a finite linear combination of tempered, invariant eigendistributions of elliptic type. Hence, by Theorem~\ref{thm:elliptic_unique}, it is itself of elliptic type.

It remains to show that
\begin{align}\label{eq:elliptic:2}
\frac{\wt{\RD}_{G,H}(X_H, \cdot)}{\RD^G(\cdot)^{1/2}} =
\gamma_\psi(\Fh)
\sum_{
\substack{
\CO_{X_H'} \in \Fs_H^{-1}(\CO_{X_H}^\st) \\
\CO_{X_H''} \in \Gam(\Fh_{\Gstarreg}(\BR))
}
}
\frac{
\Delta(X_H'', \cdot)  \widehat{i}^H(X_H', X_H'')
}{
w(X_H'')  \RD^G(\cdot)^{1/2}
}
\end{align}
is also a tempered, invariant, \( \RI(\Fg) \)-finite distribution of elliptic type.

The \( G(\BR) \)-invariance of \eqref{eq:elliptic:2} is clear. To verify temperedness, observe that
\[
\widehat{i}^H(X_H', \cdot) = \RD^H(\cdot)^{1/2} \widehat{j}^H(X_H', \cdot)
\]
is globally bounded, as shown in \cite[(1.8.3)]{beuzart2020local}, and that \( \RD^G(\cdot)^{-1/2} \) is locally integrable and of moderate growth, by \cite[(1.7.1)]{beuzart2020local}. These two facts together imply that the expression in \eqref{eq:elliptic:2} defines a tempered distribution.

In order to show that \eqref{eq:elliptic:2} is \( \RI(\Fg) \)-finite and of elliptic type, we recall the framework of Harish-Chandra semi-simple descent, following \cite[§3.2]{beuzart2020local}.

\subsubsection{Harish-Chandra semi-simple descent}

\begin{defin}\label{defin:HC-descent:1}
An open subset \( \omega \subset \Fg(\BR) \) is called \emph{completely \( G(\BR) \)-invariant} if
\begin{enumerate}
    \item \( \omega \) is \( G(\BR) \)-invariant, and
    \item for any \( X \in \omega \), the semisimple part of \( X \) in its Jordan decomposition also lies in \( \omega \).
\end{enumerate}
\end{defin}

\begin{defin}\label{defin:HC-descent:2}
Let \( \omega \subset \Fg(\BR) \) be a completely \( G(\BR) \)-invariant open subset. 
\begin{itemize}
    \item Let \( \mathcal{C}^\infty(\omega)^G \) denote the space of smooth \( G(\BR) \)-invariant functions on \( \omega \);
    \item Let \( \mathcal{D}'(\omega)^G \) denote the space of \( G(\BR) \)-invariant distributions on \( \omega \).
\end{itemize}
\end{defin}

\begin{defin}\label{defin:HC-descent:3}
Let \( X \in \Fg(\BR) \) be a semisimple element.  
An open subset \( \omega_X \subset \Fg_X(\BR) \) is called \emph{\( G \)-good} if it is completely \( G_X(\BR) \)-invariant and the map
\begin{align}\label{eq:HC-descent:2}
\omega_X \times^{G_X(\BR)} G(\BR) &\longrightarrow \Fg(\BR), \\
(Y, g) &\longmapsto g^{-1} Y g, \nonumber
\end{align}
induces an analytic isomorphism between  
\( \omega_X \times^{G_X(\BR)} G(\BR) \) and \( \omega_X^G \),  
where \( \omega_X^G \) denotes the image of the map \eqref{eq:HC-descent:2}. Here, \( \omega_X \times^{G_X(\BR)} G(\BR) \) is the quotient of  
\( \omega_X \times G(\BR) \) by the \( G_X(\BR) \)-action
\[
g_X \cdot (Y, g) = \big( g_X Y g_X^{-1}, \, g_X g \big),
\quad 
g_X \in G_X(\BR),\; (Y, g) \in \omega_X \times G(\BR).
\]
\end{defin}
From \cite[§3.2]{beuzart2020local}, an open subset 
\( \omega_X \subset \Fg_X(\BR) \) is \( G \)-good if and only if the following conditions are satisfied:
\begin{itemize}
    \item \( \omega_X \) is completely \( G_X(\BR) \)-invariant;
    \item For any \( Y \in \omega_X \), 
    \[
        \eta^G_X(Y) := \big|\det\big(\ad(Y)\big)_{\Fg / \Fg_X}\big| \neq 0;
    \]
    \item For any \( g \in G(\BR) \), the intersection 
    \( g^{-1}\omega_X g \cap \omega_X \) is nonempty if and only if 
    \( g \in G_X(\BR) \).
\end{itemize}
Let \( \omega_X \subset \Fg_X(\BR) \) be a \( G \)-good open neighborhood of \( X \), and set \( \omega = \omega_X^G \).  
Then \( \omega \) is completely \( G(\BR) \)-invariant, since \( \omega_X \) is completely \( G_X(\BR) \)-invariant.  
We have the following integration formula:
\[
\int_\omega f(Y)\, dY \;=\;
\int_{G_X(\BR) \backslash G(\BR)}
\int_{\omega_X} 
f\big(g^{-1} Y g\big)\,
\eta_X^G(Y)\, dY\, dg,
\quad f \in L^1(\omega).
\]
For every function \( f \) defined on \( \omega \), let \( f_{X,\omega_X} \) be the function on \( \omega_X \) defined by
\[
f_{X,\omega_X}(Y) := \eta_X^G(Y)^{1/2}\, f(Y).
\]
Then the map \( f \mapsto f_{X,\omega_X} \) induces the following topological isomorphisms:
\[
\mathcal{C}^\infty(\omega)^G 
\simeq \mathcal{C}^\infty(\omega_X)^{G_X},
\quad
\mathcal{C}^\infty(\omega_\reg)^G 
\simeq \mathcal{C}^\infty(\omega_{X,\reg})^{G_X},
\]
where \( \omega_\reg = \omega \cap \Fg_\reg(\BR) \) and 
\( \omega_{X,\reg} = \omega_X \cap \Fg_\reg(\BR) \).
Similarly, we have an isomorphism of spaces of invariant distributions:
\begin{align*}
\mathcal{D}'(\omega)^G &\simeq \mathcal{D}'(\omega_X)^{G_X}, \\
T &\longmapsto T_{X,\omega_X},
\end{align*}
where for \( T \in \mathcal{D}'(\omega)^G \), the distribution \( T_{X,\omega_X} \) is the unique \( G_X(\BR) \)-invariant distribution on \( \omega_X \) satisfying
\[
\langle T, f \rangle 
=
\int_{G_X(\BR)\backslash G(\BR)}
\langle 
T_{X,\omega_X},\,
({}^g f)_{X,\omega_X}
\rangle \, dg,
\quad f \in \mathcal{C}^\infty_c(\omega).
\]
Moreover, if \( f \) is a locally integrable \( G(\BR) \)-invariant function on \( \omega \), then \( f_{X,\omega_X} \) is also locally integrable and
\[
(T_f)_{X,\omega_X} = T_{\,f_{X,\omega_X}},
\]
where \( T_f \) and \( T_{\,f_{X,\omega_X}} \) denote the locally integrable distributions represented by \( f \) and \( f_{X,\omega_X} \), respectively.

For an endoscopic tuple 
\(\RE = (G, G^*, \varphi, H, s, \xi) \in \CE(\BR)\),  
note that \(\hat{G}\) and \(G\) share the same absolute Weyl group, and \(\hat{H} \subset \hat{G}\).  
By Chevalley restriction theorem, after fixing an identification of maximal tori 
\[
H \supset T_H \simeq T_G \subset G
\]
which is equivariant under the embedding of absolute Weyl groups  
\( W^H \hookrightarrow W^G \) as in Remark~\ref{rmk:admissible_embed_Weylequivariant},  
we obtain a chain of embeddings of \( \BC \)-algebras:
\[
\RI(\Fg) \;\simeq\; \BC[\Ft_G]^{W^G} 
\;\hookrightarrow\; \BC[\Ft_H]^{W^H} \;\simeq\; \RI(\Fh).
\]

Similarly, for any semisimple element \( X \in \Fg \),  
Chevalley restriction theorem provides a natural \( \BC \)-algebra embedding
\[
\RI(\Fg) \hookrightarrow \RI(\Fg_X).
\]

The following proposition is a consequence of \cite[Lemma~3.2.1]{beuzart2020local}.

\begin{pro}\label{pro:HC-descent_to_torus}
Let \( X \) be a semisimple element in \( \Fg(\BR) \), and let 
\( \omega_X \subset \Fg_X(\BR) \) be a \( G \)-good open subset.  
Set \( \omega = \omega_X^G \).  
Then for any \( T \in \mathcal{D}'(\omega)^G \) and any \( z \in \RI(\Fg) \), we have
\[
\big( \partial(z)\, T \big)_{X, \omega_X} 
= 
\big( \partial(z) \big)_{X, \omega_X} \, T_{X, \omega_X}.
\]
Moreover,
\[
\big( \partial(z) \big)_{X, \omega_X} = \partial(z_X),
\]
for all \( z \in \RI(\Fg) \),  
where \( z_X \) is the image of \( z \) under the natural injection
\[
\RI(\Fg) \hookrightarrow \RI(\Fg_X).
\]
\end{pro}

We are now ready to show that \eqref{eq:elliptic:2} is \( \RI(\Fg) \)-finite and of elliptic type.  
For any \( z \in \RI(\Fg) \), consider
\begin{equation}\label{eq:last_pf:2}
\partial(z) \bigg(
\gamma_\psi(\Fh)
\sum_{
\substack{
\CO_{X_H'} \in \Fs_H^{-1}(\CO_{X_H}^\st) \\
\CO_{X_H''} \in \Gam(\Fh_{\Gstarreg}(\BR))
}
}
\frac{
\Delta(X_H'', \cdot)\,
\widehat{i}^H(X_H', X_H'')
}{
w(X_H'')\, \RD^G(\cdot)^{1/2}
}
\bigg),
\end{equation}
where we recall that \( X_H' \) runs over the stable conjugacy class of \( X_H \) modulo \( H(\BR) \)-conjugation, and  
\( X_H'' \) runs over \( \Fh_{\Gstarreg}(\BR) \) modulo \( H(\BR) \)-conjugation.

By the definition of the transfer factor, for any \( X_G \in \Fg_\reg(\BR) \) such that  
\( \Delta(X_H'', X_G) \neq 0 \), there exists a diagram 
$
\FD(X_H'', X_G; \BR).
$
By Definition~\ref{defin:Gstarreg}, this implies that there is an isomorphism of tori (here we work with the Lie algebras) defined over \( \BR \) sending \( X_H'' \) to \( X_G \):
\begin{align}\label{eq:last_pf:1}
\CI^{X_G}_{X_H''} \colon 
\Ft_{X_H''} &\longrightarrow \Ft_{X_G}, \\
X_H'' &\longmapsto X_G. \nonumber
\end{align}

Let \( \omega_{X_G} \subset \Ft_{X_G}(\BR) \) be a \( G \)-good open subset.  
Under the isomorphism \eqref{eq:last_pf:1}, the set \( \omega_{X_G} \) is mapped to a  
\( H \)-good open subset \( \omega_{X_H''} \subset \Ft_{X_H''}(\BR) \).  
We then set
\[
\omega^G = \omega_{X_G}^G, 
\qquad
\omega^H = \omega_{X_H''}^H.
\]

We are going to study the descent of \eqref{eq:last_pf:2} to the $G$-good neighborhood $\ome_{X_G}\ni X_G$. By Proposition \ref{pro:HC-descent_to_torus},
\begin{align}\label{eq:last_pf:5}
&\Bigg(
    \partial(z)
\bigg(
\gam_\psi(\Fh)
\sum_{
\substack{
\CO_{X^\p_H}\in \Fs^{-1}_H(\CO^\st_{X_H}) 
\\ 
\CO_{X^{\p\p}_H}\in \Gam(\Fh_{\Gstarreg}(\BR))}
}
\frac{\Del(X^{\p\p}_H,\cdot)\wh{i}^H(X^\p_H,X^{\p\p}_H)}{
w(X^{\p\p}_H)
\RD^G(\cdot)^{1/2}
}
\bigg)
\Bigg)_{X_G,\ome_{X_G}}     \nonumber
\\ 
=& 
\partial(z_{X_G})
\bigg(
\gam_\psi(\Fh)
\sum_{
\substack{
\CO_{X^\p_H}\in \Fs^{-1}_H(\CO^\st_{X_H}) 
\\ 
\CO_{X^{\p\p}_H}\in \Gam(\Fh_{\Gstarreg}(\BR))}
}
\frac{\Del(X^{\p\p}_H,\cdot)\wh{i}^H(X^\p_H,X^{\p\p}_H)}{
w(X^{\p\p}_H)\RD^G(\cdot)^{1/2}}
\bigg)_{X_G,\ome_{X_G}},
\end{align}
where $z_{X_G}$ is the image of $z$ under the injection $\RI(\Fg)\hookrightarrow \RI(\Fg_{X_G}=\Ft_{X_G})$. 

By the definition of descent for functions, the above equation is equal to the following 
\begin{align*} 
=&
\partial(z_{X_G})
\bigg(
\gam_\psi(\Fh)
\sum_{
\substack{
\CO_{X^\p_H}\in \Fs^{-1}_H(\CO^\st_{X_H}) 
\\ 
\CO_{X^{\p\p}_H}\in \Gam(\Fh_{\Gstarreg}(\BR))}
}
\frac{\Del(X^{\p\p}_H,\cdot)\wh{i}^H(X^\p_H,X^{\p\p}_H)}{
w(X^{\p\p}_H)}
\bigg)\quad \text{restricted to $\ome_{X_G}\subset \Ft_{X_G}$}
\\ 
=&
\gam_\psi(\Fh)
\sum_{
\substack{
\CO_{X^\p_H}\in \Fs^{-1}_H(\CO^\st_{X_H}) 
\\ 
\CO_{X^{\p\p}_H}\in \Gam(\Fh_{\Gstarreg}(\BR))}
}
\partial(z_{X_G})
\bigg(
\frac{\Del(X^{\p\p}_H,\cdot)\wh{i}^H(X^\p_H,X^{\p\p}_H)}{
w(X^{\p\p}_H)}
\bigg)\quad \text{restricted to $\ome_{X_G}\subset \Ft_{X_G}$}.
\end{align*}
For fixed $X^{\p\p}_H\in \Fh_{\Gstarreg}(\BR)$ with attached diagram $\FD(X^{\p\p}_H,X_G;\BR)$, applying the isomorphism $\CI^{X_G}_{X^{\p\p}_H}:\Ft_{X^{\p\p}_H}\simeq \Ft_{X_G}$ which sends $X^{\p\p}_H$ to $X_G$, we get the above equation is equal to the following
\begin{align}\label{eq:last_pf:4}
=&
\gam_\psi(\Fh)
\sum_{
\substack{
\CO_{X^\p_H}\in \Fs^{-1}_H(\CO^\st_{X_H}) 
\\ 
\CO_{X^{\p\p}_H}\in \Gam(\Fh_{\Gstarreg}(\BR))}
}
\partial(z_{X^{\p\p}_H})
\bigg(
\frac{\Del\big(X^{\p\p}_H,\CI^{X_G}_{X^{\p\p}_H}(X^{\p\p}_H)\big)\wh{i}^H(X^\p_H,X^{\p\p}_H)}{
w(X^{\p\p}_H)}
\bigg)\quad \text{restricted to $\ome_{X^{\p\p}_H}\subset \Ft_{X^{\p\p}_H}$}            \nonumber
\\ 
=&
\gam_\psi(\Fh)
\sum_{
\substack{
\CO_{X^\p_H}\in \Fs^{-1}_H(\CO^\st_{X_H}) 
\\ 
\CO_{X^{\p\p}_H}\in \Gam(\Fh_{\Gstarreg}(\BR))}
}
\frac{\Del\big(X^{\p\p}_H,\CI^{X_G}_{X^{\p\p}_H}(X^{\p\p}_H)\big)
\partial(z_{X^{\p\p}_H})
\big(
\wh{i}^H(X^\p_H,X^{\p\p}_H)
\big)
}{
w(X^{\p\p}_H)}
\quad \text{restricted to $\ome_{X^{\p\p}_H}\subset \Ft_{X^{\p\p}_H}$}.
\end{align}
The last identity follows from the fact that both $w(X^{\p\p}_H)$ and 
$\Del\big(X^{\p\p}_H,\CI^{X_G}_{X^{\p\p}_H}(X^{\p\p}_H)\big) = \Del(X^{\p\p}_H,X_G)$ remain constant in a neighborhood of $X^{\p\p}_H\in \ome_{X^{\p\p}_H}$. 

For each individual term 
$$
\partial(z_{X^{\p\p}_H})
\big(
\wh{i}^H(X^\p_H,X^{\p\p}_H)
\big)\quad \text{restricted to $\ome_{X^{\p\p}_H}\subset \Ft_{X^{\p\p}_H}$,}
$$
reversing the descent to the $H$-good neighborhood $\ome_{X^{\p\p}_H}\ni X^{\p\p}_H$, the above equation becomes 
$$
\partial(z_{X^{\p\p}_H})
\big(
\wh{j}^H(X^{\p}_H, X^{\p\p}_H)
\big)_{X^{\p\p}_H, \ome_{X^{\p\p}_H}},
$$
which, by Proposition \ref{pro:HC-descent_to_torus}, is equal to 
\begin{equation}\label{eq:last_pf:3}
\big(
\partial(z_H)
\wh{j}^H(X^\p_H,X^{\p\p}_H)
\big)_{X^{\p\p}_H,\ome_{X^{\p\p}_H}},
\end{equation}
where $z_H$ is the image of $z$ under the injection $\RI(\Fg)\hookrightarrow \RI(\Fh)$. Here we use the fact that the isomorphism $\CI^{X_G}_{X^{\p\p}_H}$ is equivariant with respect to the action of Weyl groups $W^H\hookrightarrow W^G$ and the Chevalley restriction theorem, which ensures the commutativity of the following diagram. In particular, the commutativity of the diagram below implies that $z_{X^{\p\p}_H}$ is the image of $z_H$ under the natural injection $\RI(\Fh)\hookrightarrow \RI(\Ft_{X^{\p\p}_H})$.
$$
\xymatrix{
\RI(\Fg) \ar[r] \ar[d] & \RI(\Fh) \ar[d] \\
\RI(\Ft_{X_G}) \ar[r]_{\simeq}^{\CI^{X_G}_{X^{\p\p}_H}} & \RI(\Ft_{X^{\p\p}_H})
}
$$
By Lemma \ref{lem:hatJeigen}, there exists a character $\chi^H_{[X^{\p}_H]}:\RI(\Fh)\to \BC$ such that \eqref{eq:last_pf:3} is equal to 
$$
\chi^H_{[X^\p_H]}(z_H)
\big(
\wh{j}^H(X^\p_H,X^{\p\p}_H)
\big)_{X^{\p\p}_H,\ome_{X^{\p\p}_H}}.
$$
Since $X^\p_H$ is $G^*$-elliptic, there exists a diagram $\FD(X^\p_H,X^\p_G;\BR)$ with $X^\p_G$ elliptic in $\Fg_\reg(\BR)$, and there is an isomorphism $\CI^{X^\p_G}_{X^\p_H}:\Ft_{X^\p_H}\simeq \Ft_{X^\p_G}$ over $\BR$ sending $X^\p_H$ to $X^\p_G$. By the commutativity of the diagram 
$$
\xymatrix{
\RI(\Fg) \ar[r] \ar[d] & \RI(\Fh) \ar[d] \\
\RI(\Ft_{X^\p_G}) \ar[r]_{\simeq}^{\CI^{X^\p_G}_{X^{\p}_H}} & \RI(\Ft_{X^{\p}_H})
},
$$
after pulling back the character $\chi^H_{[X^\p_H]}:\RI(\Fh)\to \BC$ along the injection $\RI(\Fg)\to \RI(\Fh)$, we get the homomorphism $\chi^G_{[X^\p_G]}$ (which follows from the commutativity of the above diagram and Hilbert's Nullstellensatz). Hence \eqref{eq:last_pf:3} becomes 
$$
\chi^G_{[X^\p_G]}(z)
\big(
\wh{j}^H(X^\p_H,X^{\p\p}_H)
\big)_{X^{\p\p}_H,\ome_{X^{\p\p}_H}}.
$$
Inserting the above equation back to \eqref{eq:last_pf:4}, we deduce that \eqref{eq:last_pf:4} becomes
\begin{align*}
\gam_\psi(\Fh)
\sum_{
\substack{
\CO_{X^\p_H}\in \Fs^{-1}_H(\CO^\st_{X_H}) 
\\ 
\CO_{X^{\p\p}_H}\in \Gam(\Fh_{\Gstarreg}(\BR))
}
}
\chi^G_{[X^\p_G]}(z)
\frac{\Del\big(X^{\p\p}_H,\CI^{X_G}_{X^{\p\p}_H}(X^{\p\p}_H)\big)
\big(
\wh{i}^H(X^{\p}_H,X^{\p\p}_H)
\big)}{w(X^{\p\p}_H)}\quad
\text{restricted to $\ome_{X^{\p\p}_H}\subset \Ft_{X^{\p\p}_H}$}.
\end{align*}
Following the same argument as before, we deduce that \eqref{eq:last_pf:2} is equal to 
\begin{align*}
\gam_\psi(\Fh)
\sum_{
\substack{
\CO_{X^\p_H}\in \Fs^{-1}_H(\CO^\st_{X_H}) 
\\ 
\CO_{X^{\p\p}_H}\in \Gam(\Fh_{\Gstarreg}(\BR))
}
}
\chi^G_{[X^\p_G]}(z)
\frac{\Del\big(X^{\p\p}_H,\cdot\big)
\big(
\wh{i}^H(X^{\p}_H,X^{\p\p}_H)
\big)}{w(X^{\p\p}_H)\RD^G(\cdot)^{1/2}}.
\end{align*}
Compared with Definition \ref{defin:elliptic_type_distribution}, we deduce that $\frac{\wt{\RD}_{G,H}(X_H,\cdot)}{\RD^G(\cdot)^{1/2}}$ is indeed $\RI(\Fg)$-finite and of elliptic type. It follows that we have finished the proof of the proposition.
\end{proof}


\end{document}